\documentclass[11pt,twoside,a4paper]{article}
\usepackage{color,amscd,amsmath,amssymb,amsthm,latexsym,stmaryrd,authblk,mathabx,shuffle,hyperref,tikz,tkz-tab} 
\usepackage{mathabx}
\usepackage[latin1]{inputenc}
\usepackage[T1]{fontenc}   
\usepackage[english]{babel}
\usepackage{pgf,tikz,pgfplots}
\pgfplotsset{compat=1.15}
\usepackage{mathrsfs}
\usetikzlibrary{arrows}
\providecommand{\keywords}[1]{\textbf{\textit{Keywords.}} #1}
\providecommand{\AMSclass}[1]{\textbf{\textit{AMS classification.}} #1}

\input{xy}
\xyoption{all}

\title{Operads and bialgebras of multi-indices, and Novikov algebras}
\date{}
\author{Lo\"\i c Foissy}
\affil{\small{Univ. Littoral Côte d'Opale, UR 2597
LMPA, Laboratoire de Mathématiques Pures et Appliquées Joseph Liouville F-62100 Calais, France}.\\ Email: \texttt{foissy@univ-littoral.fr}}

\theoremstyle{plain}
\newtheorem{theo}{Theorem}[section]
\newtheorem{lemma}[theo]{Lemma}
\newtheorem{cor}[theo]{Corollary}
\newtheorem{prop}[theo]{Proposition}
\newtheorem{defi}[theo]{Definition}

\theoremstyle{remark}
\newtheorem{remark}{Remark}[section]
\newtheorem{notation}{Notations}[section]
\newtheorem{example}{Example}[section]

\newcommand{\tun}{\begin{tikzpicture}[line cap=round,line join=round,>=triangle 45,x=0.5cm,y=0.5cm]
\clip(-0.2,-0.1) rectangle (0.2,0.2);
\begin{scriptsize}
\draw [fill=black] (0.,0.) circle (1pt);
\end{scriptsize}
\end{tikzpicture}}

\newcommand{\tdeux}{\begin{tikzpicture}[line cap=round,line join=round,>=triangle 45,x=0.5cm,y=0.5cm]
\clip(-.2,-.1) rectangle (0.2,0.7);
\draw [line width=.5pt] (0.,0.5)-- (0.,0.);
\begin{scriptsize}
\draw [fill=black] (0.,0.) circle (1pt);
\draw [fill=black] (0.,0.5) circle (1pt);
\end{scriptsize}
\end{tikzpicture}}

\newcommand{\ttroisun}{\begin{tikzpicture}[line cap=round,line join=round,>=triangle 45,x=0.5cm,y=0.5cm]
\clip(-0.5,-0.1) rectangle (0.5,0.7);
\draw [line width=0.5pt] (0.,0.)-- (-0.3,0.5);
\draw [line width=0.5pt] (0.,0.)-- (0.3,0.5);
\begin{scriptsize}
\draw [fill=black] (-0.3,0.5) circle (1pt);
\draw [fill=black] (0.,0.) circle (1pt);
\draw [fill=black] (0.3,0.5) circle (1pt);
\end{scriptsize}
\end{tikzpicture}}
\newcommand{\ttroisdeux}{\begin{tikzpicture}[line cap=round,line join=round,>=triangle 45,x=0.5cm,y=0.5cm]
\clip(-.2,-.1) rectangle (0.2,1.2);
\draw [line width=0.5pt] (0.,0.5)-- (0.,0.);
\draw [line width=0.5pt] (0.,0.5)-- (0.,1.);
\begin{scriptsize}
\draw [fill=black] (0.,0.) circle (1pt);
\draw [fill=black] (0.,0.5) circle (1pt);
\draw [fill=black] (0.,1.) circle (1pt);
\end{scriptsize}
\end{tikzpicture}}

\newcommand{\tquatreun}{\begin{tikzpicture}[line cap=round,line join=round,>=triangle 45,x=0.5cm,y=0.5cm]
\clip(-0.5,-0.1) rectangle (0.5,0.7);
\draw [line width=0.5pt] (0.,0.)-- (-0.3,0.5);
\draw [line width=0.5pt] (0.,0.)-- (0.3,0.5);
\draw [line width=0.5pt] (0.,0.)-- (0.,0.5);
\begin{scriptsize}
\draw [fill=black] (-0.3,0.5) circle (1.0pt);
\draw [fill=black] (0.,0.) circle (1.0pt);
\draw [fill=black] (0.3,0.5) circle (1.0pt);
\draw [fill=black] (0.,0.5) circle (1.0pt);
\end{scriptsize}
\end{tikzpicture}}
\newcommand{\tquatredeux}{\begin{tikzpicture}[line cap=round,line join=round,>=triangle 45,x=0.5cm,y=0.5cm]
\clip(-0.5,-0.1) rectangle (0.5,1.2);
\draw [line width=0.5pt] (0.,0.)-- (-0.3,0.5);
\draw [line width=0.5pt] (0.,0.)-- (0.3,0.5);
\draw [line width=0.5pt] (-0.3,0.5)-- (-0.3,1.);
\begin{scriptsize}
\draw [fill=black] (-0.3,0.5) circle (1.0pt);
\draw [fill=black] (0.,0.) circle (1.0pt);
\draw [fill=black] (0.3,0.5) circle (1.0pt);
\draw [fill=black] (-0.3,1.) circle (1.0pt);
\end{scriptsize}
\end{tikzpicture}}
\newcommand{\tquatretrois}{\begin{tikzpicture}[line cap=round,line join=round,>=triangle 45,x=0.5cm,y=0.5cm]
\clip(-0.5,-0.1) rectangle (0.5,1.2);
\draw [line width=0.5pt] (0.,0.)-- (-0.3,0.5);
\draw [line width=0.5pt] (0.,0.)-- (0.3,0.5);
\draw [line width=0.5pt] (0.3,0.5)-- (0.3,1.);
\begin{scriptsize}
\draw [fill=black] (-0.3,0.5) circle (1.0pt);
\draw [fill=black] (0.,0.) circle (1.0pt);
\draw [fill=black] (0.3,0.5) circle (1.0pt);
\draw [fill=black] (0.3,1.) circle (1.0pt);
\end{scriptsize}
\end{tikzpicture}}
\newcommand{\tquatrequatre}{\begin{tikzpicture}[line cap=round,line join=round,>=triangle 45,x=0.5cm,y=0.5cm]
\clip(-0.5,-0.1) rectangle (0.5,1.2);
\draw [line width=0.5pt] (0.,0.)-- (0.,0.5);
\draw [line width=0.5pt] (0.,0.5)-- (0.3,1.);
\draw [line width=0.5pt] (0.,0.5)-- (-0.3,1.);
\begin{scriptsize}
\draw [fill=black] (0.,0.) circle (1.0pt);
\draw [fill=black] (0.,0.5) circle (1.0pt);
\draw [fill=black] (-0.3,1.) circle (1.0pt);
\draw [fill=black] (0.3,1.) circle (1.0pt);
\end{scriptsize}
\end{tikzpicture}}
\newcommand{\tquatrecinq}{\begin{tikzpicture}[line cap=round,line join=round,>=triangle 45,x=0.5cm,y=0.5cm]
\clip(-.2,-.1) rectangle (0.2,1.7);
\draw [line width=0.5pt] (0.,0.)-- (0.,0.5);
\draw [line width=0.5pt] (0.,0.5)-- (0.,1.);
\draw [line width=0.5pt] (0.,1.)-- (0.,1.5);
\begin{scriptsize}
\draw [fill=black] (0.,0.) circle (1.0pt);
\draw [fill=black] (0.,0.5) circle (1.0pt);
\draw [fill=black] (0.,1.) circle (1.0pt);
\draw [fill=black] (0.,1.5) circle (1.0pt);
\end{scriptsize}
\end{tikzpicture}}

\newcommand{\tcinqun}{\begin{tikzpicture}[line cap=round,line join=round,>=triangle 45,x=0.5cm,y=0.5cm]
\clip(-0.7,-0.1) rectangle (0.8,0.7);
\draw [line width=0.5pt] (0.,0.)-- (-0.5,0.5);
\draw [line width=0.5pt] (0.,0.)-- (-0.2,0.5);
\draw [line width=0.5pt] (0.,0.)-- (0.2,0.5);
\draw [line width=0.5pt] (0.,0.)-- (0.5,0.5);
\begin{scriptsize}
\draw [fill=black] (-0.5,0.5) circle (1.0pt);
\draw [fill=black] (-0.2,0.5) circle (1.0pt);
\draw [fill=black] (0.2,0.5) circle (1.0pt);
\draw [fill=black] (0.5,0.5) circle (1.0pt);
\draw [fill=black] (0.,0.) circle (1.0pt);
\end{scriptsize}
\end{tikzpicture}}
\newcommand{\tcinqdeux}{\begin{tikzpicture}[line cap=round,line join=round,>=triangle 45,x=0.5cm,y=0.5cm]
\clip(-0.5,-0.1) rectangle (0.5,1.2);
\draw [line width=0.5pt] (0.,0.)-- (-0.3,0.5);
\draw [line width=0.5pt] (0.,0.)-- (0.3,0.5);
\draw [line width=0.5pt] (0.,0.)-- (0.,0.5);
\draw [line width=0.5pt] (-0.3,0.5)-- (-0.3,1.);
\begin{scriptsize}
\draw [fill=black] (-0.3,0.5) circle (1.0pt);
\draw [fill=black] (0.,0.) circle (1.0pt);
\draw [fill=black] (0.,0.5) circle (1.0pt);
\draw [fill=black] (0.3,0.5) circle (1.0pt);
\draw [fill=black] (-0.3,1.) circle (1.0pt);
\end{scriptsize}
\end{tikzpicture}}
\newcommand{\tcinqtrois}{\begin{tikzpicture}[line cap=round,line join=round,>=triangle 45,x=0.5cm,y=0.5cm]
\clip(-0.5,-0.1) rectangle (0.5,1.2);
\draw [line width=0.5pt] (0.,0.)-- (-0.3,0.5);
\draw [line width=0.5pt] (0.,0.)-- (0.3,0.5);
\draw [line width=0.5pt] (0.,0.)-- (0.,0.5);
\draw [line width=0.5pt] (0.,0.5)-- (0.,1.);
\begin{scriptsize}
\draw [fill=black] (-0.3,0.5) circle (1.0pt);
\draw [fill=black] (0.,0.) circle (1.0pt);
\draw [fill=black] (0.,0.5) circle (1.0pt);
\draw [fill=black] (0.3,0.5) circle (1.0pt);
\draw [fill=black] (0.,1.) circle (1.0pt);
\end{scriptsize}
\end{tikzpicture}}
\newcommand{\tcinqquatre}{\begin{tikzpicture}[line cap=round,line join=round,>=triangle 45,x=0.5cm,y=0.5cm]
\clip(-0.5,-0.1) rectangle (0.5,1.2);
\draw [line width=0.5pt] (0.,0.)-- (-0.3,0.5);
\draw [line width=0.5pt] (0.,0.)-- (0.3,0.5);
\draw [line width=0.5pt] (0.,0.)-- (0.,0.5);
\draw [line width=0.5pt] (0.3,0.5)-- (0.3,1.);
\begin{scriptsize}
\draw [fill=black] (-0.3,0.5) circle (1.0pt);
\draw [fill=black] (0.,0.) circle (1.0pt);
\draw [fill=black] (0.,0.5) circle (1.0pt);
\draw [fill=black] (0.3,0.5) circle (1.0pt);
\draw [fill=black] (0.3,1.) circle (1.0pt);
\end{scriptsize}
\end{tikzpicture}}
\newcommand{\tcinqcinq}{\begin{tikzpicture}[line cap=round,line join=round,>=triangle 45,x=0.5cm,y=0.5cm]
\clip(-0.5,-0.1) rectangle (0.5,1.2);
\draw [line width=0.5pt] (0.,0.)-- (-0.3,0.5);
\draw [line width=0.5pt] (-0.3,0.5)-- (-0.3,1.);
\draw [line width=0.5pt] (0.,0.)-- (0.3,0.5);
\draw [line width=0.5pt] (0.3,0.5)-- (0.3,1.);
\begin{scriptsize}
\draw [fill=black] (-0.3,0.5) circle (1.0pt);
\draw [fill=black] (0.,0.) circle (1.0pt);
\draw [fill=black] (-0.3,1.) circle (1.0pt);
\draw [fill=black] (0.3,0.5) circle (1.0pt);
\draw [fill=black] (0.3,1.) circle (1.0pt);
\end{scriptsize}
\end{tikzpicture}}
\newcommand{\tcinqsix}{\begin{tikzpicture}[line cap=round,line join=round,>=triangle 45,x=0.5cm,y=0.5cm]
\clip(-0.7,-0.1) rectangle (0.5,1.2);
\draw [line width=0.5pt] (0.,0.)-- (-0.3,0.5);
\draw [line width=0.5pt] (0.,0.)-- (0.3,0.5);
\draw [line width=0.5pt] (-0.3,0.5)-- (-0.6,1.);
\draw [line width=0.5pt] (-0.3,0.5)-- (0.,1.);
\begin{scriptsize}
\draw [fill=black] (-0.3,0.5) circle (1.0pt);
\draw [fill=black] (0.,0.) circle (1.0pt);
\draw [fill=black] (0.3,0.5) circle (1.0pt);
\draw [fill=black] (-0.6,1.) circle (1.0pt);
\draw [fill=black] (0.,1.) circle (1.0pt);
\end{scriptsize}
\end{tikzpicture}}
\newcommand{\tcinqsept}{\begin{tikzpicture}[line cap=round,line join=round,>=triangle 45,x=0.5cm,y=0.5cm]
\clip(-0.5,-0.1) rectangle (0.8,1.2);
\draw [line width=0.5pt] (0.,0.)-- (-0.3,0.5);
\draw [line width=0.5pt] (0.,0.)-- (0.3,0.5);
\draw [line width=0.5pt] (0.3,0.5)-- (0.6,1.);
\draw [line width=0.5pt] (0.3,0.5)-- (0.,1.);
\begin{scriptsize}
\draw [fill=black] (-0.3,0.5) circle (1.0pt);
\draw [fill=black] (0.,0.) circle (1.0pt);
\draw [fill=black] (0.3,0.5) circle (1.0pt);
\draw [fill=black] (0.6,1.) circle (1.0pt);
\draw [fill=black] (0.,1.) circle (1.0pt);
\end{scriptsize}
\end{tikzpicture}}
\newcommand{\tcinqhuit}{\begin{tikzpicture}[line cap=round,line join=round,>=triangle 45,x=0.5cm,y=0.5cm]
\clip(-0.5,-0.1) rectangle (0.5,1.7);
\draw [line width=0.5pt] (0.,0.)-- (-0.3,0.5);
\draw [line width=0.5pt] (-0.3,0.5)-- (-0.3,1.);
\draw [line width=0.5pt] (0.,0.)-- (0.3,0.5);
\draw [line width=0.5pt] (-0.3,1.)-- (-0.3,1.5);
\begin{scriptsize}
\draw [fill=black] (-0.3,0.5) circle (1.0pt);
\draw [fill=black] (0.,0.) circle (1.0pt);
\draw [fill=black] (-0.3,1.) circle (1.0pt);
\draw [fill=black] (0.3,0.5) circle (1.0pt);
\draw [fill=black] (-0.3,1.5) circle (1.0pt);
\end{scriptsize}
\end{tikzpicture}}

\newcommand{\tcinqdix}{\begin{tikzpicture}[line cap=round,line join=round,>=triangle 45,x=0.5cm,y=0.5cm]
\clip(-0.5,-0.1) rectangle (0.5,1.7);
\draw [line width=0.5pt] (0.,0.)-- (0.,0.5);
\draw [line width=0.5pt] (0.,0.5)-- (0.3,1.);
\draw [line width=0.5pt] (0.,0.5)-- (-0.3,1.);
\draw [line width=0.5pt] (0.,0.5)-- (0.,1.);
\begin{scriptsize}
\draw [fill=black] (0.,0.) circle (1.0pt);
\draw [fill=black] (0.,0.5) circle (1.0pt);
\draw [fill=black] (-0.3,1.) circle (1.0pt);
\draw [fill=black] (0.3,1.) circle (1.0pt);
\draw [fill=black] (0.,1.) circle (1.0pt);
\end{scriptsize}
\end{tikzpicture}}
\newcommand{\tcinqonze}{\begin{tikzpicture}[line cap=round,line join=round,>=triangle 45,x=0.5cm,y=0.5cm]
\clip(-0.5,-0.1) rectangle (0.5,1.7);
\draw [line width=0.5pt] (0.,0.)-- (0.,0.5);
\draw [line width=0.5pt] (0.,0.5)-- (0.3,1.);
\draw [line width=0.5pt] (0.,0.5)-- (-0.3,1.);
\draw [line width=0.5pt] (-0.3,1.)-- (-0.3,1.5);
\begin{scriptsize}
\draw [fill=black] (0.,0.) circle (1.0pt);
\draw [fill=black] (0.,0.5) circle (1.0pt);
\draw [fill=black] (-0.3,1.) circle (1.0pt);
\draw [fill=black] (0.3,1.) circle (1.0pt);
\draw [fill=black] (-0.3,1.5) circle (1.0pt);
\end{scriptsize}
\end{tikzpicture}}

\newcommand{\tcinqtreize}{\begin{tikzpicture}[line cap=round,line join=round,>=triangle 45,x=0.5cm,y=0.5cm]
\clip(-0.5,-0.1) rectangle (0.5,1.7);
\draw [line width=0.5pt] (0.,0.)-- (0.,0.5);
\draw [line width=0.5pt] (0.,0.5)-- (0.,1.);
\draw [line width=0.5pt] (0.,1.)-- (-0.3,1.5);
\draw [line width=0.5pt] (0.,1.)-- (0.3,1.5);
\begin{scriptsize}
\draw [fill=black] (0.,0.) circle (1.0pt);
\draw [fill=black] (0.,0.5) circle (1.0pt);
\draw [fill=black] (0.,1.) circle (1.0pt);
\draw [fill=black] (-0.3,1.5) circle (1.0pt);
\draw [fill=black] (0.3,1.5) circle (1.0pt);
\end{scriptsize}
\end{tikzpicture}}
\newcommand{\tcinqquatorze}{\begin{tikzpicture}[line cap=round,line join=round,>=triangle 45,x=0.2cm,y=0.5cm]
\clip(-.2,-.1) rectangle (0.2,2.2);
\draw [line width=0.5pt] (0.,0.)-- (0.,0.5);
\draw [line width=0.5pt] (0.,0.5)-- (0.,1.);
\draw [line width=0.5pt] (0.,1.)-- (0.,1.5);
\draw [line width=0.5pt] (0.,1.5)-- (0.,2.);
\begin{scriptsize}
\draw [fill=black] (0.,0.) circle (1.0pt);
\draw [fill=black] (0.,0.5) circle (1.0pt);
\draw [fill=black] (0.,1.) circle (1.0pt);
\draw [fill=black] (0.,1.5) circle (1.0pt);
\draw [fill=black] (0.,2.) circle (1.0pt);
\end{scriptsize}
\end{tikzpicture}}

\newcommand{\tdun}[1]{\begin{tikzpicture}[line cap=round,line join=round,>=triangle 45,x=0.5cm,y=0.5cm]
\clip(-.2,-.1) rectangle (0.5,0.5);
\begin{scriptsize}
\draw [fill=black] (0.,0.) circle (1pt);
\end{scriptsize}
\draw(0.3,0.1) node {\tiny #1};
\end{tikzpicture}}

\newcommand{\tddeux}[2]{\begin{tikzpicture}[line cap=round,line join=round,>=triangle 45,x=0.5cm,y=0.5cm]
\clip(-.2,-.1) rectangle (0.5,1.);
\draw [line width=.5pt] (0.,0.5)-- (0.,0.);
\begin{scriptsize}
\draw [fill=black] (0.,0.) circle (1pt);
\draw [fill=black] (0.,0.5) circle (1pt);
\end{scriptsize}
\draw(0.3,0.1) node {\tiny #1};
\draw(0.3,0.6) node {\tiny #2};
\end{tikzpicture}}

\newcommand{\tdtroisun}[3]{\begin{tikzpicture}[line cap=round,line join=round,>=triangle 45,x=0.5cm,y=0.5cm]
\clip(-0.5,-0.1) rectangle (0.5,1.2);
\draw [line width=0.5pt] (0.,0.)-- (-0.3,0.5);
\draw [line width=0.5pt] (0.,0.)-- (0.3,0.5);
\begin{scriptsize}
\draw [fill=black] (-0.3,0.5) circle (1pt);
\draw [fill=black] (0.,0.) circle (1pt);
\draw [fill=black] (0.3,0.5) circle (1pt);
\end{scriptsize}
\draw(0.35,0.1) node {\tiny #1};
\draw(0.3,0.8) node {\tiny #2};
\draw(-0.3,0.8) node {\tiny #3};
\end{tikzpicture}}
\newcommand{\tdtroisdeux}[3]{\begin{tikzpicture}[line cap=round,line join=round,>=triangle 45,x=0.5cm,y=0.5cm]
\clip(-.2,-.1) rectangle (0.5,1.5);
\draw [line width=0.5pt] (0.,0.5)-- (0.,0.);
\draw [line width=0.5pt] (0.,0.5)-- (0.,1.);
\begin{scriptsize}
\draw [fill=black] (0.,0.) circle (1pt);
\draw [fill=black] (0.,0.5) circle (1pt);
\draw [fill=black] (0.,1.) circle (1pt);
\end{scriptsize}
\draw(0.3,0.1) node {\tiny #1};
\draw(0.3,0.6) node {\tiny #2};
\draw(0.3,1.1) node {\tiny #3};
\end{tikzpicture}}

\newcommand{\tdquatreun}[4]{\begin{tikzpicture}[line cap=round,line join=round,>=triangle 45,x=0.5cm,y=0.5cm]
\clip(-0.5,-0.1) rectangle (0.5,1.2);
\draw [line width=0.5pt] (0.,0.)-- (-0.3,0.5);
\draw [line width=0.5pt] (0.,0.)-- (0.3,0.5);
\draw [line width=0.5pt] (0.,0.)-- (0.,0.5);
\begin{scriptsize}
\draw [fill=black] (-0.3,0.5) circle (1.0pt);
\draw [fill=black] (0.,0.) circle (1.0pt);
\draw [fill=black] (0.3,0.5) circle (1.0pt);
\draw [fill=black] (0.,0.5) circle (1.0pt);
\end{scriptsize}
\draw(0.35,0.1) node {\tiny #1};
\draw(0.3,0.8) node {\tiny #2};
\draw(0.,0.8) node {\tiny #3};
\draw(-0.3,0.8) node {\tiny #4};
\end{tikzpicture}}
\newcommand{\tdquatredeux}[4]{\begin{tikzpicture}[line cap=round,line join=round,>=triangle 45,x=0.5cm,y=0.5cm]
\clip(-0.7,-0.1) rectangle (0.5,1.7);
\draw [line width=0.5pt] (0.,0.)-- (-0.3,0.5);
\draw [line width=0.5pt] (0.,0.)-- (0.3,0.5);
\draw [line width=0.5pt] (-0.3,0.5)-- (-0.3,1.);
\begin{scriptsize}
\draw [fill=black] (-0.3,0.5) circle (1.0pt);
\draw [fill=black] (0.,0.) circle (1.0pt);
\draw [fill=black] (0.3,0.5) circle (1.0pt);
\draw [fill=black] (-0.3,1.) circle (1.0pt);
\end{scriptsize}
\draw(0.35,0.1) node {\tiny #1};
\draw(0.3,0.8) node {\tiny #2};
\draw(-0.6,0.6) node {\tiny #3};
\draw(-0.3,1.3) node {\tiny #4};
\end{tikzpicture}}
\newcommand{\tdquatretrois}[4]{\begin{tikzpicture}[line cap=round,line join=round,>=triangle 45,x=0.5cm,y=0.5cm]
\clip(-0.5,-0.1) rectangle (0.7,1.7);
\draw [line width=0.5pt] (0.,0.)-- (-0.3,0.5);
\draw [line width=0.5pt] (0.,0.)-- (0.3,0.5);
\draw [line width=0.5pt] (0.3,0.5)-- (0.3,1.);
\begin{scriptsize}
\draw [fill=black] (-0.3,0.5) circle (1.0pt);
\draw [fill=black] (0.,0.) circle (1.0pt);
\draw [fill=black] (0.3,0.5) circle (1.0pt);
\draw [fill=black] (0.3,1.) circle (1.0pt);
\end{scriptsize}
\draw(0.35,0.1) node {\tiny #1};
\draw(0.6,0.6) node {\tiny #2};
\draw(0.3,1.3) node {\tiny #3};
\draw(-0.3,0.8) node {\tiny #4};
\end{tikzpicture}}
\newcommand{\tdquatrequatre}[4]{\begin{tikzpicture}[line cap=round,line join=round,>=triangle 45,x=0.5cm,y=0.5cm]
\clip(-0.5,-0.1) rectangle (0.5,1.7);
\draw [line width=0.5pt] (0.,0.)-- (0.,0.5);
\draw [line width=0.5pt] (0.,0.5)-- (0.3,1.);
\draw [line width=0.5pt] (0.,0.5)-- (-0.3,1.);
\begin{scriptsize}
\draw [fill=black] (0.,0.) circle (1.0pt);
\draw [fill=black] (0.,0.5) circle (1.0pt);
\draw [fill=black] (-0.3,1.) circle (1.0pt);
\draw [fill=black] (0.3,1.) circle (1.0pt);
\end{scriptsize}
\draw(0.3,0.1) node {\tiny #1};
\draw(0.3,0.5) node {\tiny #2};
\draw(0.3,1.3) node {\tiny #3};
\draw(-0.3,1.3) node {\tiny #4};
\end{tikzpicture}}

\newcommand{\K}{\mathbb{K}}
\newcommand{\N}{\mathbb{N}}
\newcommand{\Z}{\mathbb{Z}}
\newcommand{\Q}{\mathbb{Q}}

\newcommand{\g}{\mathfrak{g}}
\newcommand{\sym}{\mathfrak{S}}

\newcommand{\calP}{\mathcal{P}}

\newcommand{\PL}{\mathbf{PL}}
\newcommand{\tdelta}{\tilde{\Delta}}
\newcommand{\id}{\mathrm{Id}}
\renewcommand{\ker}{\mathrm{Ker}}
\newcommand{\vect}{\mathrm{Vect}}

\newcommand{\NMI}{\mathbf{NMI}}
\newcommand{\bfP}{\mathbf{P}}
\newcommand{\CDA}{\mathbf{CDA}}
\newcommand{\prelie}{\mathbf{PL}}
\newcommand{\nov}{\mathbf{Nov}}
\newcommand{\coinv}{\mathrm{coInv}}
\newcommand{\inv}{\mathrm{Inv}}
\newcommand{\Alg}{\K\langle X_i\mid i\in \N\rangle_+}
\newcommand{\alg}{\K[X_i\mid i\in \N]_+}
\newcommand{\algdual}{\K[x_i\mid i\in \N]_+}

\newcommand{\calT}{\mathcal{T}}
\newcommand{\im}{\mathrm{Im}}
\newcommand{\chara}{\mathrm{Char}}
\newcommand{\calF}{\mathcal{F}}
\newcommand{\HCK}{H_{\mathrm{CK}}}
\newcommand{\HGL}{H_{\mathrm{GL}}}
\newcommand{\DeltaCK}{\Delta_{\mathrm{CK}}}
\newcommand{\deltaCK}{\delta_{\mathrm{CK}}}
\newcommand{\DeltaNMI}{\Delta_{\NMI}}
\newcommand{\deltaNMI}{\delta_{\NMI}}
\newcommand{\phiCK}{\Phi_{\mathrm{CK}}}
\newcommand{\phiMI}{\Phi_{\mathrm{MI}}}
\newcommand{\Deltash}{\Delta_{\mathrm{sh}}}
\newcommand{\deltash}{\tdelta_{\mathrm{sh}}}
\newcommand{{\deltadec}}{\tdelta_{\mathrm{dec}}}
\newcommand{\muMI}{\mu_{\mathrm{MI}}}

\newcommand{\scrT}{\mathscr{T}}
\newcommand{\calU}{\mathcal{U}}

\begin{document}

\maketitle

\begin{abstract}
Noncommutative multi-indices are noncommutative monomials in a $\N$-indexed family of indeterminates. We define on them a $\Z$-graded operadic structure, with the help of a shifting derivation. Multi-indices of degree 0 are called populated: they form a suboperad, isomorphic to the operad of Novikov algebras. This operadic structure, and the relation between pre-Lie and Novikov algebras, induces two bialgebraic structure in cointeraction on commutative multi-indices. 
We show how to combinatorially embed this double bialgebra into the Connes-Kreimer Hopf algebra of rooted trees, 
with its two coproducts based, firstly on cuts, secondly, on contraction of edges, and how this embedding can be characterized by a Dyson-Schwinger equation. We also study the unique polynomial invariant compatible with the two bialgebraic structures on multi-indices and use to describe the antipode for the first coproduct. 
\end{abstract}

\\

\AMSclass{16T05 05C05 17A30}

\tableofcontents

\section*{Introduction}

Novikov algebras were introduced in the eighties in Physics \cite{Belinskii1985,Gelfand1980}. More recently, they appeared in the context of stochastic PDEs, with an underlying combinatorics of multi-indices \cite{Bruned2023-2,Linares2023,Zhu2024}. 
They are particular types of Lie algebras, where the Lie bracket is the antisymmetrization of a product $\triangleleft$, satisfying the axioms
\begin{align}
\label{EQ1}&\forall x,y,z\in A, &(x\triangleleft y)\triangleleft z-x\triangleleft(y\triangleleft z)&=(x\triangleleft z)\triangleleft y-x\triangleleft(z\triangleleft y),\\
\label{EQ2}&&x\triangleleft(y\triangleleft z)&=y\triangleleft(x\triangleleft z).
\end{align}
The first axiom (\ref{EQ1}) is the (right) pre-Lie relation, and insures by itself that the antisymmetrization of $\triangleleft$ satisfies the Jacobi identity. 
Pre-Lie,  or Gerstenhaber,  or left-symmetric algebras, introduced by Vinberg \cite{Vinberg63} and Gerstenhaber \cite{Gerstenhaber63} in the 60's,
are closely related to the combinatorics of rooted trees with graftings, as shown in \cite{Ermolaev1994,Chapoton2001}, were free pre-Lie algebras or their operad $\PL$ are described.
The second axiom (\ref{EQ2}) is the (left) non-associative permutative relation, also related to rooted trees, as shown in \cite{Livernet2006}. A typical example of Novikov algebra is given by a commutative differential algebra $(A,m,D)$ (see Definition \ref{defi2.5}),
with the Novikov product given by
\begin{align*}
&\forall a,b\in A,&a\triangleleft b&=D(a)b.
\end{align*} 
Free Novikov algebras are described in \cite{Dzumi2002}.\\

Our aim here is to explore the combinatorics of multi-indices attached to Novikov algebras, their attached operadic and double bialgebraic structure, as well as the polynomial invariants induced by the latter.
We start by the introduction of an operadic structure on words in a family of non-commuting indeterminates $X_i$, with $i\geq 0$, which we here call noncommutative multi-indices (Definition \ref{defi2.1}).
Noncommutative multi-indices of length $n$ generate the vector space $\NMI(n)$, on which the symmetric group $\sym_n$ naturally acts (on the right) by permutations of the letters. Note that 
the free (non unitary) associative algebra $\Alg$ generated by the indeterminates $X_i$'s  can be decomposed as
\[\Alg=\bigoplus_{n=1}^\infty \NMI(n).\]
We give $\Alg$ a derivation $D$ defined by
\begin{align*}
&\forall i\geq 0,&D(X_i)&=X_{i+1},
\end{align*} 
see Definition \ref{defi2.2}. The operadic structure on $\NMI$ is defined  (Proposition \ref{prop2.3}) by
\[X_{i_1}\ldots X_{i_n}\circ (P_1,\ldots,P_n)=D^{i_1}(P_1)\ldots D^{i_n}(P_n),\]
where $P_1,\ldots,P_n$ are noncommutative multi-indices. For example, if $i,j,k,l,m\geq 0$,
\begin{align*}
X_iX_j\circ (X_kX_l,X_m)&=\sum_{i=i_1+i_2}\frac{i!}{i_1!i_2!}X_{k+i_1}X_{l+i_2}X_{m+j},\\
X_iX_j\circ (X_k,X_lX_m)&=\sum_{j=j_1+j_2}\frac{j!}{j_1!j_2!}X_{k+i}X_{l+j_1}X_{m+j_2}.
\end{align*}
We prove in  Theorem \ref{theo2.8} that the operad $(\NMI,\circ)$ is generated by the elements $X_1\in \NMI(1)$ and $X_0X_0\in \NMI(2)$, with the relations
\begin{align*}
X_0X_0^{(12)}&=X_0X_0,\\
X_0X_0\circ_1 X_0X_0&=X_0X_0\circ_2 X_0X_0,\\
X_1\circ X_0X_0&=X_0X_0\circ_1X_1+X_0X_0\circ_2X_1.
\end{align*}
In other words, $(\NMI,\circ)$ is the operad of differential commutative algebras; the product is given by the action of $X_0X_0$, and the derivation by the action of $X_1$.
It turns out that this operad is $\Z$-graded (Proposition \ref{prop2.10}): if $X_1\ldots X_{i_n}$ is a noncommutative multi-index, its degree is 
\[\deg(X_{i_1}\ldots X_{i_n})=i_1+\cdots+i_n-n+1.\]
Elements of degree $0$ correspond to populated multi-indices \cite{Bruned2023-2}. Consequently, the subspaces $(\NMI_0(n))_{n\geq 1}$ of noncommutative multi-indices of degree $0$ form a suboperad,
which we prove to be isomorphic to the operad $\nov$ of Novikov algebras (Proposition \ref{prop2.19}): this isomorphism $\theta_\nov$ sends the generator $\triangleleft$ of the operad $\nov$ to $X_1X_0$. 
The operad $\nov$ is obviously a quotient of the operad $\PL$ of pre-Lie algebras: consequently, we obtain a surjective operad morphism $\theta_\PL$ from $\PL$ to $\NMI_0$, which sends the generator $\triangleleft$ of $\PL$ to $X_1X_0$. 

We then use use the results of \cite{Foissy55} to obtain pre-Lie and bialgebraic structures on $\NMI$ and its coinvariants. Firstly, the operad structure induces a pre-Lie product $\triangleleft$ on $\NMI$, see Corollary \ref{cor3.2}:
if $X_{i_1}\cdots X_{i_n}$ and $P$ are two noncommutative multi-indices,
\[X_{i_1}\cdots X_{i_n}\triangleleft P=\sum_{1\leq k \leq n}X_{i_1}\cdots X_{i_k-1}D^{i_k}(P)X_{i_k+1}\cdots X_{i_n}.\]
This structure induces a pre-Lie product, also denoted by $\triangleleft$, on the space of coinvariants of $\NMI$:
\[\coinv(\NMI)=\bigoplus_{n=1}^\infty \frac{\NMI(n)}{\vect(P-P^\sigma\mid P\in \NMI(n), \:\sigma \in \sym_n)}=\alg.\]
If $P,Q\in \alg$,
\[P\triangleleft Q=\sum_{n=0}^\infty D^n(Q) \frac{\partial P}{\partial X_n}.\]

We apply the Guin-Oudom construction \cite{Oudom2005,Oudom2008} to this pre-Lie algebra, and this allows us to describe its enveloping algebra. In particular, Proposition \ref{prop3.5} shows that if $P,P_1,\ldots,P_k\in \alg$,
then the extension of the pre-Lie product is given by
\[P\triangleleft P_1\cdots P_k=\sum_{l_1,\ldots,l_k\in \N} D^{l_1}(P_1)\cdots D^{l_k}(P_k)\frac{\partial^kP}{\partial X_{l_1}\cdots \partial X_{l_k}}.\]
Dually, we obtain a coproduct $\deltaNMI$ on the tensor algebra $T(\Alg)$, given on any $P\in \Alg$ by
\begin{align*}
\deltaNMI(P)&=\sum_{k=1}^\infty \sum_{(j_1,\ldots,j_k)\in \N^k} X_{j_1}\cdots X_{j_k}\otimes \left(\mid^{(k-1)}\circ (D'^{j_1}\otimes \cdots \otimes D'^{j_k})\circ \deltadec^{(k-1)}(P)\right),
\end{align*}
where $D'$ is the derivation of $\Alg$ defined by
\begin{align*}
&\forall i\geq 0,& D'(X_i)&=\begin{cases}
X_{i-1}\mbox{ if }i\geq 1,\\
0\mbox{ if }i=0.
\end{cases}
\end{align*}
Here, $\mid$ is the concatenation product of $T(\Alg)$, denoted in this way to avoid confusions with the product of $\Alg$, and $\deltadec$ is the reduced deconcatenation coproduct. 
For example, if $m,n,p\geq 0$, 
\begin{align*}
\deltaNMI(X_mX_nX_p)&=\sum_{i=0}^m \sum_{j=0}^n \sum_{k=0}^p\frac{(i+j+k)!}{i!j!k!}X_{i+j+k}\otimes X_{m-i}X_{n-j}X_{p-k}\\
&+\sum_{i=0}^m \sum_{j=0}^n \sum_{k=0}^p\frac{(i+j)!}{i!j!}X_{i+j}X_k\otimes X_{m-i}X_{n-j}\mid X_{p-k}\\
&+\sum_{i=0}^m \sum_{j=0}^n \sum_{k=0}^p\frac{(j+k)!}{j!k!}X_iX_{j+k}\otimes X_{m-i}\mid X_{n-j}X_{p-k}\\
&+\sum_{i=0}^m \sum_{j=0}^n \sum_{k=0}^pX_iX_jX_k\otimes X_{m-i}\mid X_{n-j}\mid X_{p-k}.
\end{align*}
This in turn induces a bialgebra structure on the quotient $S(\algdual)$.\\

A double bialgebra is a family $(B,m,\Delta,\delta)$ such that $(B,m,\delta)$ is a bialgebra
and $(B,m,\Delta)$ is a bialgebra in the category of right comodules over $(B,m,\delta)$, for the coaction given by $\delta$ itself. 
A simple example is given by the algebra of polynomials $\K[X]$, with the two coproducts defined by
\begin{align*}
\Delta(X)&=X\otimes 1+1\otimes X,&\delta(X)&=X\otimes X.
\end{align*}
Another example is given by the double structure of the Butcher-Connes-Kreimer Hopf algebra $\HCK$  \cite{Connes1998,Grossman89,Grossman90,Panaite2000,Hoffman2003,Foissy3}. 
Recall that $\HCK$ has for basis the set of rooted forests:
\[\calF=\left\{\begin{array}{c}
1,\tun,\tdeux,\tun\tun\tun,\tdeux\tun,\ttroisun,\ttroisdeux,\tun\tun\tun\tun,\tdeux\tun\tun,\tdeux\tdeux,\ttroisun\tun,\ttroisdeux\tun,\tquatreun,\tquatredeux,\tquatrequatre,\tquatrecinq,\ldots
\end{array}\right\}.\]
Its product is the disjoint union of rooted forests. Its first coproduct $\DeltaCK$ is given by admissible cuts:
\begin{align*}
\DeltaCK\left(\tquatreun\right)&=\tquatreun\otimes 1+1\otimes \tquatreun+3\ttroisun\otimes \tun+3\tdeux\otimes \tun\tun+\tun\otimes \tun\tun\tun,\\
\DeltaCK\left(\hspace{-2mm}\begin{array}{c}\tquatredeux\end{array}\hspace{-2mm}\right)&=\tquatredeux\otimes 1+1\otimes \tquatredeux+\ttroisun\otimes \tun+\ttroisdeux\otimes \tun+\tdeux\otimes \tdeux+\tdeux\otimes \tun\tun+\tun \otimes \tdeux\tun,\\
\DeltaCK\left(\hspace{-2mm}\begin{array}{c}\tquatrequatre\end{array}\hspace{-2mm}\right)&=\tquatrequatre\otimes 1+1\otimes \tquatrequatre+2\ttroisdeux\otimes \tun+\tun \otimes \ttroisun+\tdeux\otimes \tun\tun.
\end{align*}
Calaque, Ebrahimi-Fard and Manchon defined in \cite{Calaque2011} a second coproduct on $\HCK$  given by contraction and deletion of edges, making it one of the first  known examples of combinatorial double bialgebras:
\begin{align*}
\deltaCK\left(\tquatreun\right)&=\tquatreun\otimes \tun\tun\tun\tun+\tun\otimes \tquatreun+3\ttroisun\otimes \tdeux\tun\tun+3\tdeux\otimes \ttroisun\tun,\\
\deltaCK\left(\hspace{-2mm}\begin{array}{c}\tquatredeux\end{array}\hspace{-2mm}\right)&=\tquatredeux\otimes \tun\tun\tun\tun+\tun\otimes \tquatredeux+\left(2\ttroisun+\ttroisdeux\right)\otimes \tdeux\tun\tun+\tdeux\otimes\left(\ttroisun\tun+\ttroisdeux\tun+\tdeux\tdeux\right),\\
\deltaCK\left(\hspace{-2mm}\begin{array}{c}\tquatrequatre\end{array}\hspace{-2mm}\right)&=\tquatrequatre\otimes \tun\tun\tun\tun+\tun\otimes \tquatrequatre+\left(\ttroisun+2\ttroisdeux\right)\otimes \tdeux\tun\tun+\tdeux\otimes\left(\ttroisun\tun+2\ttroisdeux\tun\right).
\end{align*}
A decorated version of this double bialgebra plays an important role in the context of regularity structure \cite{Bruned2019}.

The dual of $(\HCK,m,\DeltaCK) $ is the Grossman-Larson Hopf algebra  $\HGL$, also based on rooted forests, and which product is given by graftings. For example,
\begin{align*}
\tdeux*\tun\tun&=\tdeux\tun\tun+2\ttroisun\tun+2\ttroisdeux\tun+\tquatreun+2\tquatredeux+\tquatrequatre,&\tun\tun*\tdeux&=\tdeux\tun\tun+2\ttroisdeux\tun.
\end{align*} 
The Hopf algebra $\HGL$ can also be seen as the Guin-Oudom extension applied to the free pre-Lie algebra on one generator, as described in \cite{Chapoton2001,Ermolaev1994}. 
Other combinatorial examples based on posets, graphs, partitions or others can be found in the literature \cite{Ebrahimi-Fard2022,Manchon2012,Foissy27,Foissy34,Foissy36,Foissy44,Foissy45}.
As explained in \cite{Foissy55}, the operad morphism $\theta_\PL$ from $\PL$ to $\NMI$ induces a second coproduct $\DeltaNMI$ on $S(\algdual)$, making $(S(\algdual),\mid,\DeltaNMI,\deltaNMI)$ a double bialgebra. This coproduct is described in Proposition \ref{prop3.14}:
\begin{align*}
&\forall P\in \algdual,&\DeltaNMI(P)&=P\otimes 1+1\otimes P+\sum_{k=1}^\infty \frac{1}{k!}(D'^k\otimes \mid^{(k-1)})\circ \deltash^{(k)}(P).
\end{align*}
Up to minor rescalings and re-indexations\footnote{In particular, in \cite{Zhu2024}, the indices of the indeterminates start at $-1$, whereas they start at $0$ here.}, this is the coproduct on multi-indices described in \cite{Zhu2024}. 
For example, if $i,j,k\geq 0$,
\begin{align*}
\DeltaNMI(x_ix_jx_k)&=x_ix_jx_k\otimes 1+1\otimes x_ix_jx_k+x_{i-1}\otimes x_jx_k+x_{j-1}\otimes x_ix_k+x_{k-1}\otimes x_ix_j\\
&+x_{j-1}x_k\otimes x_i+x_jx_{k-1}\otimes x_i+x_{i-1}x_k\otimes x_j+x_ix_{k-1}\otimes x_j+x_{i-1}x_j\otimes x_k\\
&+x_ix_{j-1}\otimes x_k+x_{i-2}\otimes x_j\mid x_k+x_{j-2}\otimes x_i\mid x_k+x_{k-2}\otimes x_i\mid x_j,
\end{align*}
with the convention that $x_l=0$ if $l<0$. By functoriality of this construction, the morphism $\theta_\PL$
induces a double bialgebra morphism $\Psi$ from $(S(\algdual),\mid,\DeltaNMI,\deltaNMI)$ to $(\HCK,m,\DeltaCK,\deltaCK)$, which we describe in Theorem \ref{theo3.16} and Proposition \ref{prop3.18}:
to any rooted tree $T$, we associate a monomial $M(T)=\displaystyle \prod_{i=0}^\infty x_i^{n_i(T)}$, where $n_i(T)$ is the number of vertices of $T$ of fertility $i$. For any monomial $x^\alpha$ of $\algdual$, 
\[\Psi(x^\alpha)=c_\alpha\left(\sum_{T\in \calT,\: M(T)=x^\alpha} p_TT\right),\]
where $c_\alpha$ is a certain coefficient (see Proposition \ref{prop3.18}), and $p_T$ is the number of embeddings of $T$ in the plane.
For example,
\begin{align*}
\Psi(x_1^3x_0)&=\tquatreun,&
\Psi(x_2x_1x_0^2)&=2\tquatredeux+\tquatrequatre,\\
\Psi(x_1^3x_0)&=6\tquatrecinq,&
\Psi(x_4x_0^3)&=\tcinqun,\\
\Psi(x_3x_1x_0^3)&=3\tcinqdeux+\tcinqdix,&
\Psi(x_2^2x_0^3)&=6\tcinqsix,\\
\Psi(x_2x_1^2x_0^2)&=4\tcinqhuit+2\tcinqcinq+4\tcinqonze+2\tcinqtreize,&
\Psi(x_1^4x_0)&=24\tcinqquatorze.
\end{align*}
In particular, if $x^\alpha$ is not of degree $0$, then $\Psi(x^\alpha)=0$. 
We show in Proposition \ref{prop3.19} and Theorem \ref{theo3.21} how to generate these elements thanks to a fixed-point equation, similar to combinatorial Dyson-Schwinger equations used in Quantum Field Theory \cite{Bergbauer2006,Kreimer2007,Kreimer2008,Kreimer2018,Tanasa2013,Balduf2024,Hihn2019,Marie2013,Yeats2011,Foissy14}. \\

A particularly interesting result is, for a given double bialgebra $(B,m,\Delta,\delta)$,  the existence and unicity of a double bialgebra morphism to the double bialgebra $(\K[X],m,\Delta,\delta)$, under a condition of connectivity  (see Section \ref{sect1} for details). 
For the double bialgebra $\HCK$, for example, this morphism send any rooted tree $T$ to a polynomial $\phiCK(T)$, which, when evaluated in $n\in \N$, gives the number of strictly increasing maps from the set of vertices of $T$ 
(ordered with the partial relation induced by the tree structure) to $\{1,\ldots,n\}$. For example, for any $n\in \N$,
\begin{align*}
\phiCK\left(\hspace{-2mm}\begin{array}{c}\ttroisdeux\end{array}\hspace{-2mm}\right)(n)&=|\{(a,b,c)\in \{1,\ldots,n\}^3\mid a<b<c\}|=\frac{n(n-1)(n-2)}{6},\\
\phiCK\left(\hspace{-2mm}\begin{array}{c}\ttroisun\end{array}\hspace{-2mm}\right)(n)&=|\{(a,b,c)\in \{1,\ldots,n\}^3\mid a<b,c\}|=\frac{n(n-1)(2n-1)}{6}.
\end{align*}
We denote by $\phiMI$ the double bialgebra morphism  from $(S(\algdual),\mid,\DeltaNMI,\deltaNMI)$ to $(\K[X],m,\Delta,\delta)$. 
We give in Proposition \ref{prop4.2} a fixed-point equation satisfied by the formal series
\[\calP=\sum_{\alpha \in \Lambda} \phiMI(x^\alpha)\dfrac{X^\alpha}{\alpha!} \in \K[X][[X_i\mid i\in \N]],\]
which allows to compute $\phiMI(x^\alpha)$ by induction on the length of $\alpha$. For example,
\begin{align*}
\phiMI(x_4x_0^4)&=\frac{1}{30}(3X^2-3X-1)(2X-1)(X-1)X,\\
\phiMI(x_3x_1x_0^3)&=\frac{1}{120}(42X^2-39X-1)(X-1)(X-2)X,\\
\phiMI(x_2^2x_0^3)&=\frac{1}{20}(8X^2-11X+1)(X-1)(X-2)X,\\
\phiMI(x_2x_1^2x_0^2)&=\frac{1}{60}(11X-29)(2X-1)(X-1)(X-2)X,\\
\phiMI(x_1^4x_0)&=\frac{1}{5}(X-1)(X-2)(X-3)(X-4)X.
\end{align*}
These polynomials (more specifically, their values in $-1$) are used to give a formula for the antipode of $(S(\algdual),\mid,\DeltaNMI)$ in Proposition \ref{prop4.5}. \\

This paper is organized as follows.
The first section contains remainders, on double bialgebras (definitions and main results), on the pre-Lie operad $\PL$, and on the Connes-Kreimer and Grossman-Larson Hopf algebras. 
The operad of multi-indices $\NMI$ is introduced in the next section, where we also give a presentation by generators and relations, then detail a suboperad $\NMI_0$ isomorphic to the operad of Novikov algebras 
and combinatorially describe the operad morphism $\theta_\PL$ from $\PL$ to $\NMI_0$. 
The induced algebraic structures are studied in the third section, from a pre-Lie algebra to the dual double bialgebra via the Guin-Oudom procedure. We also describe the double bialgebra morphism from multi-indices to rooted trees in this section,
with the help of a combinatorial Dyson-Schwinger-like equation. The final section is devoted to consequences: we describe the fundamental polynomial invariant associated to multi-indices and shows how to compute it thanks to a fixed-point equation, 
and details applications to the antipode. \\

\textbf{Notations} 
\begin{enumerate}
\item We denote by $\K$ a commutative field of characteristic zero. All the vector spaces of this text will be taken over this field.
\item For any $n\in \N$, we denote by $[n]$ the set $\{1,\ldots,n\}$. In particular, $[0]=\emptyset$.
\item Let $(H,m,\Delta)$ be a bialgebra, of counit $\varepsilon$. We denote by $H_+$ the kernel of $\varepsilon$. Then $H_+$ is given a coassociative, non necessarily counitary coproduct $\tdelta$, defined by
\begin{align*}
&\forall x\in H_+,&\tdelta(x)&=\Delta(x)-x\otimes 1-1\otimes x.
\end{align*}
We denote by $\tdelta^{(n)}:H_+\longrightarrow H_+^{\otimes (n+1)}$ the iterations of $\tdelta$:
\begin{align*}
\tdelta^{(n)}&=\begin{cases}
\id_{H_+}\mbox{ if }n=0,\\
(\tdelta^{(n-1)}\otimes \id_{H_+})\circ \tdelta \mbox{ if }n\geq 1.
\end{cases}\end{align*}
We shall say that $H$ is connected if $\tdelta$ is locally nilpotent: for any $x\in H_+$, there exists $n\geq 0$ such that $\tdelta^{(n)}(x)=0$.
\end{enumerate}

\section{Reminders}

\label{sect1}

\subsection{Double bialgebras}

We refer to \cite{Foissy37,Foissy36,Foissy40} for the details.

\begin{defi}
A double bialgebra is a family $(H,m,\Delta,\delta)$ such that:
\begin{enumerate}
\item $(H,m,\Delta)$ and $(H,m,\delta)$ are bialgebras. Their common unit is denoted by $1_H$.
The counits of $\Delta$ and $\delta$ are respectively denoted by $\varepsilon_\Delta$ and $\epsilon_\delta$. We put
\[\eta_b:\left\{\begin{array}{rcl}
\K&\longrightarrow&H\\
\lambda&\longmapsto&\lambda 1_H.
\end{array}\right.\]
\item $(H,m,\Delta)$ is a bialgebra in the category of right comodules over $(H,m,\delta)$, with the coaction $\delta$, seen as a coaction over itself. This is equivalent to the two following assertions:
\begin{align*}
(\varepsilon_\Delta \otimes \id_H)\circ \delta&=\eta_H\circ \varepsilon_\Delta,\\
(\Delta \otimes \id_H)\circ \delta&=m_{1,3,24}\circ (\delta\otimes \delta)\circ \Delta,
\end{align*}  
where
\[m_{1,3,24}:\left\{\begin{array}{rcl}
H^{\otimes 4}&\longrightarrow& H^{\otimes 3}\\
x_1\otimes x_2\otimes x_3\otimes x_4&\longmapsto&x_1\otimes x_3\otimes x_2x_4.\\
\end{array}\right.\]
\end{enumerate}\end{defi}

\begin{example}\label{ex1.1}
An example of double bialgebra is given by the usual polynomial algebra $\K[X]$, with its usual product $m$ and the two (multiplicative) coproducts defined by
\begin{align*}
\Delta(X)&=X\otimes 1+1\otimes X,&\delta(X)&=X\otimes X.
\end{align*}
The counits are given by
\begin{align*}
\varepsilon_\Delta:&\left\{\begin{array}{rcl}
\K[X]&\longrightarrow&\K\\
P(X)&\longmapsto&P(0),
\end{array}\right.&
\epsilon_\delta:&\left\{\begin{array}{rcl}
\K[X]&\longrightarrow&\K\\
P(X)&\longmapsto&P(1).
\end{array}\right.\end{align*}
 We identify $\K[X]\otimes \K[X]$ with $\K[X,Y]$, trough the algebra morphism
\[\left\{\begin{array}{rcl}
\K[X]\otimes \K[X]&\longrightarrow&\K[X,Y]\\
P(X)\otimes Q(X)&\longmapsto&P(X)Q(Y).
\end{array}\right.\]
Then, for any $P\in \K[X]$, 
\begin{align*}
\Delta(P)(X,Y)&=P(X+Y),&\delta(P)(X,Y)&=P(XY). 
\end{align*}\end{example}

\begin{prop}\label{prop1.2}
Let $(H,m,\Delta,\delta)$ be a double bialgebra. 
\begin{enumerate}
\item We denote by $\chara(H)$ the set of characters of $H$, that is to say the set of algebra morphisms from $H$ to $\K$. This sets inherits two associative and unitary products defined by
\begin{align*}
&\forall \lambda,\mu\in \chara(H),&\lambda*\mu&=(\lambda \otimes \mu)\circ \Delta,&\lambda\star \mu&=(\lambda \otimes \mu)\circ \delta.
\end{align*}
The units of the products $*$ and $\star$ are respectively $\varepsilon_\Delta$ and $\epsilon_\delta$.
\item Let $(A,m,\Delta)$ be a bialgebra. We denote by $M_{H\rightarrow A}$ the set of bialgebra morphisms from $(H,m,\Delta)$ to $(A,m,\Delta)$. 
Then the monoid $(\chara(H),\star)$ acts on $M_{H\rightarrow A}$ via the right action given by
\begin{align*}
\leftsquigarrow&:\left\{\begin{array}{rcl}
M_{H\rightarrow A}\times \chara(H)&\longrightarrow&M_{H\rightarrow A}\\
(\phi,\lambda)&\longmapsto&\phi\leftsquigarrow\lambda=(\phi\otimes \lambda)\circ \delta.
\end{array}\right.
\end{align*}\end{enumerate}\end{prop}

\begin{remark}
The compatibility between $\Delta$ and $\delta$ implies that for any $\lambda,\mu,\nu \in \chara(H)$,
\[(\lambda *\mu)\star \nu=(\lambda \star \nu)*(\mu\star \nu).\]
\end{remark}

The double structure allows to find the antipode for the first structure, whenever it exists:

\begin{theo}\label{theo1.3}  \cite[Corollary 2.3]{Foissy40}
Let $(H,m,\Delta,\delta)$ be a double bialgebra. 
\begin{enumerate}
\item Then $(H,m,\Delta)$ is a Hopf algebra if, and only if, the character $\epsilon_\delta$ has an inverse $\mu_H$ for the convolution product $*$ dual to $\Delta$. 
Moreover, if this holds, the antipode of $(H,m,\Delta)$ is given by
\[S=(\mu_H \otimes \id_H)\circ \delta.\]
\item Let $\phi_H:H\longrightarrow \K[X]$ be a double bialgebra morphism. Then $\epsilon_\delta$ has an inverse for the convolution product $*$, given by
\begin{align*}
\mu_H&: \left\{\begin{array}{rcl}
H&\longrightarrow&\K\\
x&\longmapsto&\phi_H(x)(-1).
\end{array}\right.
\end{align*}\end{enumerate}\end{theo}

We shall say that a double bialgebra $(H,m,\Delta,\delta)$ is connected if the bialgebra $(H,m,\Delta)$ is connected. If so, we obtain more results:

\begin{theo}\label{theo1.4} \cite[Theorem 3.9, Corollary 3.11]{Foissy40}
Let $(H,m,\Delta,\delta)$ be a connected double bialgebra.
A polynomial invariant over $H$ is a bialgebra map from $(H,m,\Delta)$ to $(\K[X],m,\Delta)$.
The set of polynomial invariants overs $H$ is denoted by $M_{H\longrightarrow \K[X]}$.
\begin{enumerate}
\item There exists a unique double bialgebra morphism $\phi_H$ from $H$ to $\K[X]$,
which will be called the fundamental polynomial invariant. Moreover,
\begin{align*}
&\forall x\in H_+,&\phi_H(x)&=\sum_{k=1}^\infty \epsilon_\delta^{\otimes k}\circ \tdelta^{(k-1)}(x) \dfrac{X(X-1)\ldots (X-k+1)}{k!}.
\end{align*}
\item The two following maps are bijective, inverse one from the other:
\begin{align*}
&\left\{\begin{array}{rcl}
\chara(H)&\longrightarrow&M_{H\rightarrow \K[X]}\\
\lambda&\longmapsto&\phi_H\leftsquigarrow \lambda,
\end{array}\right.
&&\left\{\begin{array}{rcl}
M_{H\rightarrow \K[X]}&\longrightarrow&\chara(H)\\
\phi&\longmapsto&\left\{\begin{array}{rcl}
H&\longrightarrow&\K\\
x&\longmapsto&\phi(x)(1).
\end{array}\right.
\end{array}\right.
\end{align*}
\end{enumerate}\end{theo}

\subsection{Indexed trees and the Pre-Lie operad}

A (right) pre-Lie algebra is a pair $(\g,\blacktriangleleft)$, such that $\g$ is a vector space and $\blacktriangleleft$ is a bilinear product on $\g$ satisfying the pre-Lie relation:
\begin{align*}
&\forall x,y,z\in \g,&(x\blacktriangleleft y)\blacktriangleleft z-x\blacktriangleleft(y\blacktriangleleft z)&=(x\blacktriangleleft z)\blacktriangleleft y-x\blacktriangleleft(z\blacktriangleleft y).
\end{align*}

\begin{example}\begin{enumerate}
\item Any associative algebra is pre-Lie.
\item Let $\lambda \in \K$. The Faà di Bruno pre-Lie algebra $\g_{\mathrm{FdB}(\lambda)}$ has a basis $(e_n)_{n\geq 1}$ and its pre-Lie product is defined by
\begin{align*}
&\forall k,l\geq 1,& e_k \blacktriangleleft e_l&=(k+\lambda)e_{k+l}.
\end{align*}\end{enumerate}\end{example} 

We refer to \cite{Dotsenko2016,Loday2012,Yau2016} for the notations and main results on operads.
The operad of pre-Lie algebras is denoted by $\prelie$. It is generated by $\blacktriangleleft\in \prelie(2)$ and the relation
\begin{align*}
\blacktriangleleft\circ_1\blacktriangleleft-\blacktriangleleft\circ_2\blacktriangleleft&=(\blacktriangleleft\circ_1 \blacktriangleleft-\blacktriangleleft\circ_2\blacktriangleleft)^{(23)}.
\end{align*} 
We shall use the formalism of rooted trees from \cite{Chapoton2001}  to combinatorially describe the operad $\prelie$. 
For any $n\geq 1$, $\prelie(n)$ is generated by the set of $n$-indexed rooted trees, or in other words rooted trees which set of vertices is $[n]$. For example,
\begin{align*}
\prelie(1)&=\vect(\tdun{$1$}),&
\prelie(2)&=\vect(\tddeux{$1$}{$2$},\tddeux{$2$}{$1$}),&
\prelie(3)&=\vect\left(\begin{array}{c}\tdtroisun{$1$}{$3$}{$2$},\tdtroisun{$2$}{$3$}{$1$},\tdtroisun{$3$}{$2$}{$1$},
\tdtroisdeux{$1$}{$2$}{$3$},\tdtroisdeux{$1$}{$3$}{$2$},\tdtroisdeux{$2$}{$1$}{$3$},\tdtroisdeux{$2$}{$3$}{$1$},\tdtroisdeux{$3$}{$1$}{$2$},\tdtroisdeux{$3$}{$2$}{$1$}
\end{array}\right).\end{align*}
The operadic composition is given by insertion at vertices, see \cite{Chapoton2001} for details. For example,
\begin{align*}
\tddeux{$1$}{$2$}\circ\left(\tdtroisun{$1$}{$3$}{$2$},\tdun{$1$}\right)&=\tdquatreun{$1$}{$4$}{$3$}{$2$}+\tdquatredeux{$1$}{$3$}{$2$}{$4$}+\tdquatretrois{$1$}{$3$}{$4$}{$2$},&
\tddeux{$1$}{$2$}\circ\left(\tdun{$1$},\tdtroisun{$1$}{$3$}{$2$}\right)&=\tdquatrequatre{$1$}{$2$}{$4$}{$3$}.
\end{align*}

\begin{notation}
 \label{not1.1} We denote by $c_n$ the $n$-th corolla, that is to say the $n$-indexed tree whose root is the vertex 1 and the leaves the vertices $2,\ldots,n$:
\vspace{-5mm}
\begin{align*}
c_1&=\tdun{$1$},&c_2&=\tddeux{$1$}{$2$},&c_3&=\tdtroisun{$1$}{$3$}{$2$},&c_4&=\tdquatreun{$1$}{$4$}{$3$}{$2$}\ldots
\end{align*}
\end{notation}

Let us recall the Guin-Oudom construction \cite{Oudom2005,Oudom2008}. If $(\g,\blacktriangleleft)$ is a pre-Lie algebra, 
the symmetric algebra $S(\g)$ is given its usual product $m$ and coproduct $\Delta$. We extend the pre-Lie product of $\g$ to $S(\g)$ by
\begin{align*}
&\forall x\in S(\g),&1\blacktriangleleft x&=\varepsilon_\Delta(x),\\
&\forall x,x_1,\ldots,x_n \in \g,&x\blacktriangleleft x_1\cdots x_n&=c_{n+1}\cdot(x,x_1,\ldots,x_n),\\
&\forall x,y,z\in S(\g),& (xy)\blacktriangleleft z&=\sum \left(x\blacktriangleleft z^{(1)}\right) \left(y\blacktriangleleft z^{(2)}\right),
\end{align*}
where we used Sweedler's notation $\displaystyle \Delta(z)=\sum z^{(1)}\otimes z^{(2)}$,
and where $\cdot$ represents the action of $\PL$ on $\g$. We then define the product $*$ by
\begin{align*}
&\forall x,y\in S(g),&x*y&=\sum y^{(1)}\left(x\blacktriangleleft y^{(2)}\right).
\end{align*}
Then $(S(\g),*,\Delta)$ is a Hopf algebra, isomorphic to the enveloping algebra of the Lie algebra $\g$, with the Lie bracket defined by
\begin{align*}
&\forall x,y\in \g,&[x,y]&=x\blacktriangleleft y-y\blacktriangleleft x.
\end{align*}

\begin{example}
When applied to the Faà di Bruno pre-Lie algebra $\g_{\mathrm{FdB}(\lambda)}$, this gives
\begin{align*}
&\forall k,l_1,\ldots,l_n \geq 1,&e_k \blacktriangleleft e_{l_1}\cdots e_{l_m}&=(k+\lambda)k\cdots (k-(n-2)\lambda)e_{k+l_1+\cdots+l_n}.
\end{align*}
The enveloping algebra $\calU(\g_{\mathrm{FdB}(\lambda)})$ is a graded Hopf algebra. When $\lambda=1$, the group of characters of its graded dual is isomorphic to the group of formal diffeomorphisms tangent to the identity:
\[G_{\mathrm{FdB}}=(\{X+a_1X^2+a_2X^3+\cdots \in \K[[X]]\},\circ).\]
\end{example}

Particular examples of pre-Lie algebras are given by brace algebras, see \cite{Ronco2000,Ronco2001} for a precise definition. A brace algebra $\g$ is given a family of operators $\{-;-\}:\g\otimes \g^{\otimes n}\longrightarrow \g$ for any $n\geq 1$,
with certain conditions we won't detail here. In particular, it is a pre-Lie algebra, with
\begin{align*}
&\forall x,y\in \g,&x\blacktriangleleft y&=\{x;y\}.
\end{align*}
The Guin-Oudom extension of the pre-Lie product can be described with the help of the brace structure:
\begin{align*}
&\forall x,x_1,\ldots,x_n \in \g,&x\blacktriangleleft x_1\cdots x_n&=\sum_{\sigma \in \sym_n} \{x;x_{\sigma(1)},\ldots, x_{\sigma(n)}\}.
\end{align*}

\subsection{Connes-Kreimer and Grossman-Larson Hopf algebras}

The set of rooted trees is denoted by $\calT$:
\[\calT=\left\{\begin{array}{c}
\tun,\tdeux,\ttroisun,\ttroisdeux,\tquatreun,\tquatredeux,\tquatrequatre,\tquatrecinq,\tcinqun,\tcinqdeux,\tcinqcinq,\tcinqsix,\tcinqhuit,\tcinqdix,\tcinqonze,\tcinqtreize,\tcinqquatorze,\ldots
\end{array}\right\}.\]
The vector space generated by $\calT$ is denoted by $\g_\calT$. It is the free pre-Lie algebra generated by $\tun$, with the pre-Lie product $\blacktriangleleft$ defined by the sum of graftings at any vertex \cite{Chapoton2001,Ermolaev1994}. For example,
\begin{align*}
\tun \blacktriangleleft \tdeux&=\ttroisdeux,&\tdeux \blacktriangleleft\tun&=\ttroisun+\ttroisdeux,&\tdeux \blacktriangleleft \tdeux&=\tquatredeux+\tquatrecinq,\\
\ttroisun\blacktriangleleft \tun&=\tquatreun+2\tquatredeux,&\tun \blacktriangleleft \ttroisun&=\tquatrequatre,&
\ttroisdeux \blacktriangleleft \tun&=\tquatredeux+\tquatrequatre+\tquatrecinq,&\tun \blacktriangleleft \ttroisdeux&=\tquatrecinq.
\end{align*}
The Guin-Oudom construction can be done on this pre-Lie algebra, giving the Grossman-Larson Hopf algebra $\HGL$ \cite{Grossman89,Grossman90}.
As a vector space, $\HGL=S(\g_\calT)$ has for basis the set $\calF$ of rooted forests, that is to say commutative monomials in rooted trees:
\[\calF=\left\{\begin{array}{c}
1,\tun,\tdeux,\tun\tun\tun,\tdeux\tun,\ttroisun,\ttroisdeux,\tun\tun\tun\tun,\tdeux\tun\tun,\tdeux\tdeux,\ttroisun\tun,\ttroisdeux\tun,\tquatreun,\tquatredeux,\tquatrequatre,\tquatrecinq,\ldots
\end{array}\right\}.\]
Its coproduct is the deshuffling of forests: 
\begin{align*}
&\forall T_1,\ldots,T_k\in \calT,&\Delta_{GL}(T_1\ldots T_k)&=\sum_{I\subseteq [k]} \prod_{i\in I}T_i \otimes \prod_{i\notin I} T_i.
\end{align*} 
The Guin-Oudom product $*$ is given by extended graftings. For example,
\begin{align*}
\tun *\tdeux&=\tdeux\tun+\ttroisdeux,&\tdeux*\tun&=\tdeux\tun+\ttroisun+\ttroisdeux,\\
\tdeux*\tdeux&=\tdeux\tdeux+\tquatredeux+\tquatrecinq,&\tdeux*\tun\tun&=\tdeux\tun\tun+2\ttroisun\tun+2\ttroisdeux\tun+\tquatreun+2\tquatredeux+\tquatrequatre,&
\tun\tun*\tdeux&=\tdeux\tun\tun+2\ttroisdeux\tun.
\end{align*} 

\begin{notation}
For any rooted forest $F$, we denote by $s_F$ the number of symmetries of $F$, that is to say of graph automorphisms of $F$ fixing the set of roots.
\begin{align*}
\begin{array}{|c||c|c|c|c|c|c|c|c|c|c|c|c|}
\hline F&1&\tun&\tdeux&\tun\tun&\ttroisun&\ttroisdeux&\tun\tun\tun&\tdeux\tun&\tquatreun&\tquatredeux&\tquatrequatre&\tquatrecinq\\
\hline\hline&&&&&&&& \\[-4mm]
s_F&1&1&1&2&2&1&6&1&6&1&2&1\\
\hline \end{array}\end{align*}
\end{notation}

The Hopf algebra $\HGL$ is naturally graded by the number of vertices of rooted forests. We identify its graded dual with $\HGL$ itself, with the pairing $\langle-,-\rangle$ defined by
\begin{align*}
&\forall F,G\in \calF,&\langle F,G\rangle&=s_F\delta_{F,G}.
\end{align*}
Its graded dual is the Connes-Kreimer Hopf algebra $\HCK$ \cite{Connes1998,Foissy3,Hoffman2003,Panaite2000}. Its product is the disjoint union of forests
and its coproduct can be described in terms of admissible cuts. For example,
\begin{align*}
\DeltaCK\left(\ttroisun\right)&=\ttroisun\otimes 1+1\otimes \ttroisun+2\tdeux\otimes \tun+\tun \otimes \tun\tun,\\
\DeltaCK\left(\hspace{-2mm}\begin{array}{c}\ttroisdeux\end{array}\hspace{-2mm}\right)&=\ttroisdeux\otimes 1+1\otimes \ttroisdeux+\tun \otimes \tdeux+\tdeux\otimes \tun,\\
\DeltaCK\left(\tquatreun\right)&=\tquatreun\otimes 1+1\otimes \tquatreun+3\ttroisun\otimes \tun+3\tdeux\otimes \tun\tun+\tun\otimes \tun\tun\tun,\\
\DeltaCK\left(\hspace{-2mm}\begin{array}{c}\tquatredeux\end{array}\hspace{-2mm}\right)&=\tquatredeux\otimes 1+1\otimes \tquatredeux+\ttroisun\otimes \tun+\ttroisdeux\otimes \tun+\tdeux\otimes \tdeux+\tdeux\otimes \tun\tun+\tun \otimes \tdeux\tun,\\
\DeltaCK\left(\hspace{-2mm}\begin{array}{c}\tquatrequatre\end{array}\hspace{-2mm}\right)&=\tquatrequatre\otimes 1+1\otimes \tquatrequatre+2\ttroisdeux\otimes \tun+\tun \otimes \ttroisun+\tdeux\otimes \tun\tun,\\
\DeltaCK\left(\hspace{-2mm}\begin{array}{c}\tquatrecinq\end{array}\hspace{-3mm}\right)&=\tquatrecinq\otimes 1+1\otimes \tquatrecinq+\ttroisdeux\otimes \tun+\tdeux\otimes \tdeux+\tun\otimes \ttroisdeux.
\end{align*}
Its counit is given by
\begin{align*}
&\forall F\in \calF,&\varepsilon_\Delta(F)&=\delta_{F,1}.
\end{align*}

A particularly important operator on $\HCK$ is the grafting operator $B^+$, which send any rooted forest $F$ to the rooted tree obtained by grafting all the roots of $F$ on a common root. For example,
\begin{align*}
B^+(\tun\tun\tun)&=\tquatreun,&B^+\left(\tdeux\tun\right)&=\tquatredeux,&B^+\left(\ttroisun\right)&=\tquatrequatre,&B^+\left(\hspace{-2mm}\begin{array}{c}\ttroisdeux\end{array}\hspace{-2mm}\right)&=\tquatrecinq.
\end{align*}
It is a 1-cocycle for the (dual of the) Hochschild cohomology \cite{Connes1998}, that is to say
\begin{align*}
&\forall x\in \HCK,&\DeltaCK\circ B^+(x)&=1\otimes B^+(x)+(B^+\otimes \id_{\HCK})\circ \DeltaCK(x).
\end{align*}
The pair $(\HCK,B^+)$ satisfies the following universal property:

\begin{theo}[\textbf{Universal property of $\HCK$}]\label{theo1.5} \cite{Connes1998}
Let $A$ be a commutative algebra and $L:A\longrightarrow A$ be a linear map. There exists a unique algebra map $\phi:\HCK\longrightarrow A$ such that $\phi\circ B^+=L\circ \phi$.
If $A$ is a bialgebra and $L$ is a 1-cocycle of $A$, then $\phi$ is a bialgebra map.
\end{theo}

The Connes-Kreimer algebra also has a second coproduct $\deltaCK$, described in \cite{Calaque2011}, and which can be obtained from operadic considerations involving the operad $\prelie$,
see  \cite{Foissy55} for details. This second coproduct $\deltaCK$ is described by contraction and extraction of subtrees. For example,
\begin{align*}
\deltaCK(\tun)&=\tun\otimes \tun,\\
\deltaCK\left(\tdeux\right)&=\tdeux\otimes \tun\tun+\tun \otimes \tdeux,\\
\deltaCK\left(\ttroisun\right)&=\ttroisun\otimes \tun\tun\tun+\tun\otimes \ttroisun+2\tdeux\otimes \tdeux\tun,\\
\deltaCK\left(\hspace{-2mm}\begin{array}{c}\ttroisdeux\end{array}\hspace{-2mm}\right)&=\ttroisdeux\otimes \tun\tun\tun+\tun\otimes \ttroisdeux+2\tdeux\otimes \tdeux\tun,\\
\deltaCK\left(\tquatreun\right)&=\tquatreun\otimes \tun\tun\tun\tun+\tun\otimes \tquatreun+3\ttroisun\otimes \tdeux\tun\tun+3\tdeux\otimes \ttroisun\tun,
\end{align*}
\begin{align*}
\deltaCK\left(\hspace{-2mm}\begin{array}{c}\tquatredeux\end{array}\hspace{-2mm}\right)&=\tquatredeux\otimes \tun\tun\tun\tun+\tun\otimes \tquatredeux+\left(2\ttroisun+\ttroisdeux\right)\otimes \tdeux\tun\tun+\tdeux\otimes\left(\ttroisun\tun+\ttroisdeux\tun+\tdeux\tdeux\right),\\
\deltaCK\left(\hspace{-2mm}\begin{array}{c}\tquatrequatre\end{array}\hspace{-2mm}\right)&=\tquatrequatre\otimes \tun\tun\tun\tun+\tun\otimes \tquatrequatre+\left(\ttroisun+2\ttroisdeux\right)\otimes \tdeux\tun\tun+\tdeux\otimes\left(\ttroisun\tun+2\ttroisdeux\tun\right),\\
\deltaCK\left(\hspace{-2mm}\begin{array}{c}\tquatrecinq\end{array}\hspace{-3mm}\right)&=\tquatrecinq\otimes \tun\tun\tun\tun+\tun\otimes \tquatrecinq+3\ttroisdeux\otimes \tdeux\tun\tun+\tdeux\otimes \left(2\ttroisdeux\tun+\tdeux\tdeux\right).
\end{align*}
Its counit  is given by
\begin{align*}
&\forall F\in \calF,&\epsilon_{\deltaCK}(F)&=\begin{cases}
1\mbox{ if }F=\tun^n,\mbox{ for a certain }n\in \N,\\
0\mbox{ otherwise}.
\end{cases}\end{align*}

With its product and its two coproducts, $(\HCK,m,\DeltaCK,\deltaCK)$ is a double bialgebra.
The bialgebra $(\HCK,m,\DeltaCK)$ is graded by the number of vertices, and connected. Consequently, there exists a unique double bialgebra morphism
$\phiCK:(\HCK,m,\DeltaCK,\deltaCK)\longrightarrow (\K[X],m,\Delta,\delta)$. Let us describe this morphism. 

\begin{prop}
For any $P\in \K[X]$, let us denote by $L(P)\in \K[X]$ the unique polynomial such that
\begin{align*}
&\forall n\geq 1,&L(P)(n)&=P(0)+\cdots+P(n-1).
\end{align*}
Then:
\begin{enumerate}
\item $L$ is a 1-cocycle of $\K[X]$:
\begin{align*}
&\forall P\in \K[X],&\Delta \circ L(P)&=1\otimes L(P)+(L\otimes \id_{\K[X]})\circ \Delta(P).
\end{align*}
\item $L$ is a Rota-Baxter operator of weight 1 of $\K[X]$:
\begin{align*}
&\forall P,Q\in \K[X],&L(P)L(Q)&=L(L(P)Q)+L(PL(Q))+L(PQ).
\end{align*}
\item $\phiCK \circ B^+=L\circ \phiCK$.
\end{enumerate}
\end{prop}

\begin{proof}
We identify $\K[X]\otimes \K[X]$ with $\K[X,Y]$, as in Example \ref{ex1.1}.\\

1. Let $k,l\geq 1$. For any $P\in \K[X]$,
\begin{align*}
\left(1\otimes L(P)+(L\otimes \id_{\K[X]})\circ \Delta(P)\right)(k,l)&=P(0)+\cdots+P(l-1)+P(0+l)+\cdots+P(k-1+l)\\
&=P(0)+\cdots+P(k+l-1)\\
&=L(P)(k+l)\\
&=\Delta \circ L(P)(k,l).
\end{align*}
As this is true for any $k,l\geq 1$, $L$ is a 1-cocycle.\\

2. Let $n\geq 1$. 
\begin{align*}
L(P)L(Q)(n)&=\sum_{0\leq i,j\leq n-1}P(i)Q(j)\\
&=\sum_{0\leq i<j\leq n-1} P(i)P(j)+\sum_{0\leq j<i\leq n-1} P(i)P(j)+\sum_{0\leq i \leq n-1} P(i)P(i)\\
&=\sum_{0\leq j\leq n-1} L(P)(j)P(j)+\sum_{0\leq i\leq n-1} P(i)L(Q)(i)+\sum_{0\leq i\leq n-1} (PQ)(i)\\
&=L(L(P)Q)(n)+L(PL(Q))(n)+L(PQ)(n).
\end{align*}
As this is true for any $n\geq 1$, $L$ is a Rota-Baxter operator of weight 1.\\

3. By the universal property of $\HCK$ (Theorem \ref{theo1.5}), there exists a unique bialgebra morphism $\phiCK':(\HCK,m,\DeltaCK,\deltaCK)\longrightarrow (\K[X],m,\Delta,\delta)$
such that $\phiCK'\circ B^+=L\circ \phiCK'$. In order to prove that $\phiCK=\phiCK'$, it is enough to prove that $\epsilon_\delta\circ \phiCK'=\epsilon_\delta$. 
As $\phiCK'$ and $\epsilon_\delta$ are algebra morphisms, by Theorem \ref{theo1.4} (third item), it is enough to prove that for any rooted tree $T$,  
\[\phiCK'(T)(1)=\epsilon_\delta \circ \phiCK'(T)=\epsilon_{\deltaCK}(T)=\delta_{T,\tun}.\]
As $\phiCK'(\tun)=L(1)=X$, this is true if $T=\tun$. Otherwise, we put $T=B^+(F)$, with $\varepsilon(F)=0$.
Then $\phiCK'(T)=L\circ \phiCK'(F)$ and $\phiCK'(F)\in \ker(\varepsilon_\Delta)$, so $\phiCK'(F)(0)=0$. Therefore,
\[\phiCK'(T)(1)=L\circ \phiCK'(F)(1)=\phiCK'(F)(0)=0,\]
so indeed $\phiCK'=\phiCK$. 
\end{proof}

\begin{remark}
Let us consider the basis of Hilbert polynomials $(H_n(X))_{n\geq 0}$ of $\K[X]$:
\begin{align*}
&\forall n\geq 0,&H_n(X)&=\frac{X(X-1)\ldots (X-n+1)}{n!}.
\end{align*}
Then for any $n\geq 0$, $L(H_n(X))=H_{n+1}(X)$. Indeed, by Vandermonde's identity, for any $n\geq 1$,
\begin{align*}
H_k(0)+\cdots+H_k(n-1)&=\binom{0}{k}+\cdots+\binom{n-1}{k}=\binom{n}{k+1}=H_{k+1}(n).
\end{align*}\end{remark}

\begin{notation}
For any $n\geq 1$, we denote by $C_n=B^+(\tun^{n-1})$ the $n$-th corolla and by $L_n={B^+}^n(1)$ the $n$-th ladder:
\begin{align*}
C_1&=\tun,&C_2&=\tdeux,&C_3&=\ttroisun,&C_4&=\tquatreun,&C_5&=\tcinqun\ldots\\
L_1&=\tun,&L_2&=\tdeux,&L_3&=\ttroisdeux,&L_4&=\tquatrecinq,&L_5&=\tcinqquatorze\ldots 
\end{align*}\end{notation}

\begin{example}\begin{enumerate}
\item For any $n\geq 1$, \begin{align*}
\phiCK(L_n)&=L^n(1)=H_n(X),&\phiCK(C_n)&=L(X^{n-1}).
\end{align*}
In other words, for any $n\geq 1$, for any $k\geq 1$,
\[\phiCK(C_n)(k)=1^{n-1}+\cdots+k^{n-1}.\]
Therefore, the polynomials $\phiCK(C_n)$ are related to Faulhaber's formula. 
\item Here are more examples.
\begin{align*}
\phiCK(\tun)&=X,&\phiCK\left(\tdeux\right)&=\dfrac{X(X-1)}{2},\\
\phiCK\left(\ttroisun\right)&=\dfrac{X(X-1)(2X-1)}{6},&\phiCK\left(\hspace{-2mm}\begin{array}{c}\ttroisdeux\end{array}\hspace{-2mm}\right)&=\dfrac{X(X-1)(X-2)}{6},\\
\phiCK(\tquatreun)&=\dfrac{X^2(X-1)^2}{4},&\phiCK\left(\hspace{-2mm}\begin{array}{c}\tquatredeux\end{array}\hspace{-2mm}\right)&=\dfrac{(3X-1)(X-1)(X-2)X,}{24},\\
\phiCK\left(\hspace{-2mm}\begin{array}{c}\tquatrequatre\end{array}\hspace{-2mm}\right)&=\dfrac{(X-1)^2(X-2)X}{12},&\phiCK\left(\hspace{-2mm}\begin{array}{c}\tquatrecinq\end{array}\hspace{-2mm}\right)&=\dfrac{X(X-1)(X-2)(X-3)}{24}.
\end{align*}\end{enumerate}\end{example}

\section{The operad of multi-indices}

\subsection{Construction}

\begin{defi} \label{defi2.1}
Let $(X_i)_{i\in \N}$ be a family of (non commuting) indeterminates. A noncommutative multi-index is a noncommutative monomial (or a word) $X_{i_1}\cdots X_{i_n}$ in these indeterminates. 
Its length, weight and degree are respectively defined by 
\begin{align*}
\ell(X_{i_1}\cdots X_{i_n})&=n,\\
\omega(X_{i_1}\cdots X_{i_n})&=i_1+\cdots+i_n,\\
\deg(X_{i_1}\cdots X_{i_n})&=\omega(X_{i_1}\cdots X_{i_n})-\ell(X_{i_1}\cdots X_{i_n})+1=i_1+\cdots+i_n-n+1.
\end{align*}
(Note that the degree of a noncommutative monomial is not necessarily nonnegative). We denote by $\NMI(n)$ the space generated by the set of noncommutative multi-indices of length $n$. 
\end{defi}

\begin{remark}
Observe that 
\[\bigoplus_{n=1}^\infty \NMI(n)=\Alg.\]
With the concatenation product, it is the free non unitary noncommutative algebra in the indeterminates $X_i$, $i\in \N$. 
\end{remark}

\begin{defi}\label{defi2.2}
We denote by $D$ the derivation of the algebra $\Alg$ which sends the indeterminate $X_i$ to $X_{i+1}$ for any $i\in \N$. In other words, for any noncommutative multi-index $X_{i_1}\cdots X_{i_n}$,
\[D(X_{i_1}\cdots X_{i_n})=\sum_{k=1}^n X_{i_1}\cdots X_{i_{k-1}}X_{i_k+1}X_{i_{k+1}}\cdots X_{i_n}.\]
The derivation $D$ is homogeneous of degree $0$ for the length, and of degree $1$ for the weight and for the degree.
\end{defi}

\begin{prop}\label{prop2.3}
We define an operadic composition $\circ$ on $(\NMI(n))_{n\geq 1}$ as follows:
if $X_{i_1}\cdots X_{i_n}$ is a noncommutative multi-index of length $n$ and $P_1,\ldots,P_n$ are noncommutative multi-indices, then
\[X_{i_1}\cdots X_{i_n}\circ (P_1,\ldots,P_n)=D^{i_1}(P_1)\cdots D^{i_n}(P_n).\]
The unit is $X_0\in \NMI(1)$. 
\end{prop}

\begin{proof}
Using Leibniz rule for $D$, we obtain that for any noncommutative multi-indices $X_{i_1}\cdots X_{i_n}$, $X_{j_{1,1}}\cdots X_{j_{1,l_1}}$, $\ldots$,
$X_{j_{n,1}}\cdots X_{j_{n,l_n}}$,
\begin{align}
\nonumber &X_{i_1}\cdots X_{i_n}\circ (X_{j_{1,1}}\cdots X_{j_{1,l_1}}, \ldots,X_{j_{n,1}}\cdots X_{j_{n,l_n}})\\
\label{EQ3}&=\sum_{\substack{i_{1,1}+\cdots+i_{1,l_1}=i_1,\\\hspace{14mm} \vdots \\i_{n,1}+\cdots+i_{n,l_n}=i_n}}
\frac{i_1!\ldots i_n!}{i_{1,1}!\ldots i_{n,l_n}!} X_{j_{1,1}+i_{1,1}}\cdots X_{j_{n,l_n}+i_{n,l_n}}.
\end{align}

Let us prove the associativity of $\circ$. Let $X_{i_1}\cdots X_{i_n}$, $X_{j_{1,1}}\cdots X_{j_{1,l_1}}$, $\ldots$,
$X_{j_{n,1}}\cdots X_{j_{n,l_n}}$, and $P_{1,1},\ldots,P_{n,l_n}$ be noncommutative multi-indices. By (\ref{EQ3}),
\begin{align*}
&(X_{i_1}\cdots X_{i_n}\circ (X_{j_{1,1}}\cdots X_{j_{1,l_1}}, \ldots,X_{j_{n,1}}\cdots X_{j_{n,l_n}}))\circ (P_{1,1},\ldots,P_{n,l_n})\\
&=D^{i_1}(X_{j_{1,1}}\cdots X_{j_{1,l_1}})\cdots D^{i_n}(X_{j_{n,1}}\cdots X_{j_{n,l_n}})\circ (P_{1,1},\ldots,P_{n,l_n})\\
&=\sum_{\substack{i_{1,1}+\cdots+i_{1,l_1}=i_1,\\ \hspace{14mm} \vdots  \\i_{n,1}+\cdots+i_{n,l_n}=i_n}}
\frac{i_1!\ldots i_n!}{i_{1,1}!\ldots i_{n,l_n}!} X_{j_{1,1}+i_{1,1}}\ldots X_{j_{n,l_n}+i_{n,l_n}}\circ (P_{1,1},\ldots,P_{n,l_n})\\
&=\sum_{\substack{i_{1,1}+\cdots+i_{1,l_1}=i_1,\\ \hspace{14mm} \vdots  \\i_{n,1}+\cdots+i_{n,l_n}=i_n}}
\frac{i_1!\ldots i_n!}{i_{1,1}!\ldots i_{n,l_n}!} D^{j_{1,1}+i_{1,1}}(P_{1,1})\cdots D^{j_{n,l_n}+i_{n,l_n}}(P_{n,l_n}),
\end{align*}
whereas
\begin{align*}
&X_{i_1}\cdots X_{i_n}\circ (X_{j_{1,1}}\cdots X_{j_{1,l_1}}\circ (P_{1,1},\ldots,P_{1,l_1}), \ldots,X_{j_{n,1}}\cdots X_{j_{n,l_n}}\circ (P_{n,1},\ldots,P_{n,l_n}))\\
&=X_{i_1}\cdots X_{i_n}\circ (D^{j_{1,1}}(P_{1,1})\cdots D^{j_{1,l_1}}(P_{1,l_1}),\ldots,D^{j_{n,1}}(P_{n,1})\cdots D^{j_{n,l_n}}(P_{n,l_n}))\\
&=D^{i_1}(D^{j_{1,1}}(P_{1,1})\cdots D^{j_{1,l_1}}(P_{1,l_1}))\cdots D^{i_n}(D^{j_{n,1}}(P_{n,1})\cdots D^{j_{n,l_n}}(P_{n,l_n}))\\
&=\sum_{\substack{i_{1,1}+\cdots+i_{1,l_1}=i_1,\\ \hspace{14mm} \vdots  \\i_{n,1}+\cdots+i_{n,l_n}=i_n}}
\frac{i_1!\ldots i_n!}{i_{1,1}!\ldots i_{n,l_n}!} D^{j_{1,1}+i_{1,1}}(P_{1,1})\cdots D^{j_{n,l_n}+i_{n,l_n}}(P_{n,l_n}).
\end{align*}
So $\circ$ is associative. Let $X_{i_1}\cdots X_{i_n}$ be a noncommutative multi-index. As $D^i(X_0)=X_i$ for any $i\in \N$,
\[X_{i_1}\cdots X_{i_n}(X_0,\ldots,X_0)=D^{i_1}(X_0)\cdots D^{i_n}(X_0)=X_{i_1}\cdots X_{i_n}.\]
Moreover, for any $P\in \NMI(n)$, $X_0\circ P=D^0(P)=P$, so $X_0$ is the unit for the composition $\circ$.\\

Let $X_{i_1}\cdots X_{i_n}$ be a noncommutative multi-index of length $n$ and $\sigma \in \sym_n$. We put
\[(X_{i_1}\cdots X_{i_n})^\sigma=X_{i_{\sigma(1)}}\cdots X_{i_{\sigma(n)}}.\]
This defines a right action of $\sym_n$ on $\NMI(n)$. Let us show that it is compatible with the operadic composition $\circ$.
Let $X_{i_1}\cdots X_{i_n}$ be a noncommutative multi-index of length $n$, $\sigma \in \sym_n$, $P_1,\ldots,P_n$ be noncommutative multi-indices and $\tau_1,\ldots,\tau_n$ in the appropriate symmetric groups.
\begin{align*}
(X_{i_1}\cdots X_{i_n})^\sigma\circ (P_1^{\tau_1},\ldots,P_n^{\tau_n})&=X_{i_{\sigma(1)}}\cdots X_{i_{\sigma(n)}}\circ (P_1^{\tau_1},\ldots,P_n^{\tau_n})\\
&=D^{i_{\sigma(1)}}(P_1^{\tau_1})\cdots D^{i_{\sigma(n)}}(P_n^{\tau_n})\\
&=D^{i_{\sigma(1)}}(P_1)^{\tau_1}\cdots D^{i_{\sigma(n)}}(P_n)^{\tau_n}\\
&=(D^{i_1}(P_{\sigma^{-1}(1)})\cdots D^{i_n}(P_{\sigma^{-1}(n)}))^{\sigma\circ(\tau_1,\ldots,\tau_n)}\\
&=X_{i_1}\cdots X_{i_n}\circ (P_{\sigma^{-1}(1)},\cdots,P_{\sigma^{-1}(n)})^{\sigma\circ(\tau_1,\ldots,\tau_n)}.
\end{align*}
We observed for the third equality that $D$ commutes with the action of the symmetric groups. So the action of the symmetric groups and the operadic composition are compatible. 
\end{proof}

\begin{example}
Let $i,j,k,l,m\in \N$.
\begin{align*}
X_i\circ X_j&=X_{i+j},\\
X_iX_j\circ (X_k,X_l)&=X_{i+k}X_{j+l},\\
X_iX_j\circ (X_kX_l,X_m)&=\sum_{i=i_1+i_2}\frac{i!}{i_1!i_2!}X_{k+i_1}X_{l+i_2}X_{m+j},\\
X_iX_j\circ (X_k,X_lX_m)&=\sum_{j=j_1+j_2}\frac{j!}{j_1!j_2!}X_{k+i}X_{l+j_1}X_{m+j_2}.
\end{align*}
In particular, the associative algebra $(\NMI(1),\circ)$ is isomorphic to $\K[X]$. 
\end{example}

\subsection{Presentation of $\NMI$}

\begin{lemma}\label{lem2.4}
The operad $\NMI$ is generated by $X_1\in \NMI(1)$ and $X_0X_0\in \NMI(2)$.
\end{lemma}

\begin{proof}
Let $\bfP$ be the suboperad of $\NMI$ generated by $X_1$ and $X_0X_0$. Let us prove that $\bfP$ contains all noncommutative multi-indices. Firstly, observe that $X_0\in \bfP$, as it is the unit of $\NMI$. 
Moreover, for any $i\in \N$, $X_1\circ X_i=D(X_i)=X_{i+1}$. An easy induction proves that $X_i\in \bfP$ for any $i\geq 0$. Secondly, observe that for any $n\geq 1$,
\[X_0X_0\circ (X_0^n,X_0)=D^0(X_0^n)D^0(X_0)=X_0^{n+1}.\]
An easy induction proves that $X_0^n\in \bfP(n)$ for any $n\geq 1$. Let now $X_{i_1}\ldots X_{i_n}$ be a noncommutative multi-index. 
\[X_0^n \circ (X_{i_1},\ldots,X_{i_n})=D^0(X_{i_1})\cdots D^0(X_{i_n})=X_{i_1}\cdots X_{i_n} \in \bfP(n).\]
So $\bfP=\NMI$. 
\end{proof}

\begin{defi}\label{defi2.5}
A commutative associative differential algebra is a triple $(A,m,\partial)$ such that:
\begin{itemize}
\item $(A,m)$ is a commutative associative algebra, not necessarily unitary.
\item $\partial:A\longrightarrow A$ is a derivation of $(A,m)$:
\begin{align*}
&\forall a,b\in A,&\partial(ab)&=\partial(a)b+a\partial(b).
\end{align*}
\end{itemize}
We denote by $\CDA$ the operad of commutative associative algebras. In other words, $\CDA$ is generated by $\partial\in \CDA(1)$ and $m\in \CDA(2)$, with the relations
\begin{align*}
m^{(12)}&=m,&m\circ_1 m&=m\circ_2 m,&\partial\circ m&=m\circ_1 \partial+m\circ_2 \partial.
\end{align*}\end{defi}

\begin{prop}
Let us fix $N\geq 1$. We consider the (non unitary) polynomial algebra $A_N=\K[X_{i,j}\mid i\in [N], j\in \N]_+$. We denote by $D$ the differential of $A_N$ defined by
\begin{align*}
&\forall i\in [N],\: \forall n\in \N,&D(X_{i,n})&=X_{i,n+1}.
\end{align*}
This commutative differential algebra satisfies the following universal property: if $(B,m,\partial)$ is a commutative differential algebra, if $b_i\in B$ for any $i\in [N]$,
then there exists a unique algebra morphism $\phi:A_N\longrightarrow B$ such that
 $\phi \circ D=\partial \circ \phi$ and for any $i\in [N]$, $\phi(X_{i,0})=b_i$.
\end{prop}

\begin{proof}
\textit{Unicity}. For any $i\in [N]$, for any $n\in \N$, $X_{i,n}=D^n(X_{i,0})$. Therefore, $\phi$ is the unique algebra morphism such that $\phi(X_{i,n})=\partial^n (b_i)$ for any $i\in [N]$ and any $n\in \N$.\\

\textit{Existence}. Let $\phi:A_N\longrightarrow B$ be the unique algebra morphism such that $\phi(X_{i,n})=\partial^n (b_i)$ for any $i\in [N]$ and any $n\in \N$.
The universal property of the polynomial algebra $A_N$ implies the existence of $\phi$. Let us now prove that $\phi\circ D=\partial \circ \phi$. For this, we introduce
\[A'_N=\{a\in A_N\mid \phi\circ D(a)=\partial \circ \phi(a)\}.\]
As both $D$ and $\partial$ are derivations, and as $\phi$ is an algebra morphism, if $a,b\in A'_N$,
\begin{align*}
\phi \circ D(ab)&=\phi(D(a)b+aD(b))\\
&=\phi\circ D(a)\phi(b)+D(a)\phi \circ D(b)\\
&=\partial \circ \phi(a)\phi(b)+\phi(a)\partial \circ \phi(b)\\
&=\partial(\phi(a)\phi(b))\\
&=\partial \circ \phi(ab),
\end{align*}
so $ab\in A'_N$: $A'_N$ is a subalgebra of $A_N$. Let $i\in [N]$ and $n\in \N$. 
\begin{align*}
\phi \circ D(X_{i,n})&=\phi(X_{i,n+1})=\partial^{n+1}(b_i)=\partial(\partial^n(b_i))=\partial \circ \phi(X_{i,n}).
\end{align*}
So $X_{i,n}\in A'_N$. As a consequence, $A'_N=A_N$, and $\phi\circ D=\partial \circ \phi$. 
\end{proof}

In other words, $A_N$ is the free differential commutative algebra generated by the $N$ elements $X_{i,0}$, for $i\in [N]$. 
We put $V_N=\vect(X_{i,0}\mid i\in [N])$. Then $A_N$ is isomorphic to the differential commutative algebra $F_\CDA(N)$, whose underlying space is 
\[F_\CDA(N)=\bigoplus_{n=1}^\infty \CDA(n)\otimes_{\sym_n}V_N^{\otimes n}.\]
 As $X_{1,0}\otimes\cdots \otimes X_{N,0}$ generates a free $\sym_N$-module, we obtain:

\begin{lemma}
The following map is injective:
\[\iota_N:\left\{\begin{array}{rcl}
\CDA(N)&\longrightarrow&A_N\\
p&\longmapsto&p\cdot(X_{1,0}\otimes\cdots \otimes X_{N,0}).
\end{array}\right.\]
\end{lemma}

An efficient way to describe the image of $\iota_N$ by giving $A_N$ an $\N^N$-graduation: we define it by putting $X_{i,0}$ homogeneous of degree $(0,\ldots,0,1,0,\ldots,0)$ 
(the $1$ is in position $i$) for any $i\in [N]$, the product and the differential $D$ of $A_N$ being homogeneous of degree $(0,\ldots,0)$. 
Then $\iota_N(\CDA(N))$ is the component of $A_N$ homogeneous of degree $(1,\ldots,1)$, which is 
\[\iota_N(\CDA(N))=\vect(X_{1,i_1}\cdots X_{N,i_N}\mid i_1,\ldots,i_N\in \N).\] 
This space is isomorphic with $\NMI(N)$, through the linear map
\[j_N:\left\{\begin{array}{rcl}
\iota_N(\CDA(N))&\longrightarrow&\NMI(N)\\
X_{1,i_1}\cdots X_{N,i_N}&\longmapsto&X_{i_1}\cdots X_{i_N}.
\end{array}\right. \]

\begin{theo}\label{theo2.8}
There exists a unique operad isomorphism $\theta:\CDA\longrightarrow \NMI$, sending $\partial$ to $X_1$ and $m$ to $X_0X_0$. 
\end{theo}

\begin{proof}
Firstly,
\begin{align*}
X_0X_0^{(12)}&=X_0X_0,\\
X_0X_0\circ_1X_0X_0&=X_0X_0\circ_2X_0X_0=X_0X_0X_0,\\
X_1\circ X_0X_0&=X_0X_0\circ_1X_1+X_0X_0\circ_2X_1=X_1X_0+X_0X_1.
\end{align*}
Therefore, there exists a unique operad morphism $\theta:\CDA\longrightarrow \NMI$, sending $\partial$ to $X_1$ and $m$ to $X_0X_0$. 
By Lemma \ref{lem2.4}, it is surjective. \\

Let us fix $N\geq 1$. Let $X_{1,i_1}\cdots X_{N,i_N}\in \iota_N(\CDA(N))$, and let $p\in \CDA(N)$ its unique antecedent by $\iota_N$. Observe that
\[X_{1,i_1}\cdots X_{N,i_N}=D^{i_1}(X_{1,0})\cdots D^{i_N}(X_{N,0}),\]
so $p=m^{(N-1)}\circ (\partial^{i_1},\ldots,\partial^{i_1})$, where $m^{(N-1)}\in \CDA(N)$ is the $N-1$-th iteration of $m$:
\[m^{(N)}=\begin{cases}
I \mbox{ if }n=0,\\
m\mbox{ if }n=1,\\
m\circ_1 m^{(N-1)} \mbox{ if }n\geq 2.
\end{cases}\]
A direct induction proves that $\theta\left(m^{(N-1)}\right)=X_0^N$ for any $n\geq 1$. Therefore, 
\[\theta(p)=X_0^N\circ(X_{i_1},\ldots, X_{i_N})=X_{i_1}\cdots X_{i_N}.\]
We obtain that $\theta\circ \iota_N^{-1}=j_N$, which gives $\theta_{\mid \CDA(N)}=j_N\circ \iota_N$. 
As both $\iota_N$ and $j_N$ are injective, $\theta_{\mid \CDA(N)}$ is injective and finally $\theta$ is an isomorphism. 
\end{proof}

\subsection{Graduation of $\NMI$}

\begin{defi}
For any $n\in \N$ and $k\in \Z$, we denote by $\NMI_k(n)$ the subspace of $\NMI(n)$ of noncommutative multi-indices of length $n$ and degree $k$,
or equivalently of noncommutative multi-indices of length $n$ and of weight $k+n-1$.
\end{defi}

\begin{prop} \label{prop2.10}
With this decomposition, $\NMI$ is a graded operad:
\begin{itemize}
\item The unit $I=X_0$ of $\NMI$ belongs to $\NMI_0(1)$. 
\item For any $k,k_1,\ldots,k_n\in \Z$, for any $P\in \NMI_k(n)$, $P_i\in \NMI_{k_i}(m_i)$,
\[P\circ (P_1,\ldots,P_n)\in \NMI_{k+k_1+\cdots+k_n}(m_1+\cdots+m_n).\]
\item For any $k\in \Z$, for any $P\in \NMI_k(n)$, for any $\sigma \in \sym_n$, $P^\sigma \in \NMI_k(n)$. 
\end{itemize}
\end{prop}

\begin{proof}
Firstly, observe that if $Q$ and $R$ are noncommutative multi-indices, then $D(Q)$ is a sum of noncommutative multi-indices of degree $\deg(Q)+1$
and $QR$ is a noncommutative multi-index of degree $\deg(Q)+\deg(R)-1$. Therefore, putting $P=X_{i_1}\cdots X_{i_n}$, then $\deg(P)=i_1+\cdots+i_n-n+1$ and
\begin{align*}
\deg(P\circ (P_1,\ldots,P_n))&=\deg(D^{i_1}(P_1)\cdots D^{i_n}(P_n))\\
&=\deg(D^{i_1}(P_1))+\cdots+\deg(P^{i_n}(P_n))-n+1\\
&=\deg(P_1)+i_1+\cdots+\deg(P_n)+i_n-n+1\\
&=\deg(P)+\deg(P_1)+\cdots+\deg(P_n). 
\end{align*}
The action of $\sigma$ on a noncommutative multi-index obviously does not change its degree. 
\end{proof}

\begin{cor}
The family of subspaces $\NMI_0=(\NMI_0(n))_{n\geq 1}$ is a suboperad of $\NMI$. 
\end{cor}

\begin{prop}
We consider the formal series
\[f_\NMI(X,Y)=\sum_{n=1}^\infty \sum_{k=-\infty}^\infty \dim(\NMI_k(n))Y^kX^n\in \Q[Y^{-1},Y]][X]].\]
Then
\[f_\NMI(X,Y)=\sum_{n=1}^\infty \sum_{k=1-n}^\infty \binom{2n+k-2}{n-1}Y^k X^n=\frac{XY}{Y-Y^2-X}.\]
\end{prop}

\begin{proof}
For any $n\geq 1$ and $k\in \Z$, $\dim(\NMI_k(n))$ is the number of sequences $(i_1,\ldots,i_n)\in \N^n$ such that $i_1+\cdots+i_n-n+1=k$, or equivalently such that $i_1+\cdots+i_n=k+n-1$. 
This number is $\displaystyle \binom{n+k+n-1-1}{n-1}=\binom{2n+k-2}{n-1}$. In particular, it is zero if $k<1-n$. Hence, for any $n\geq 1$,
\begin{align*}
\sum_{k=-\infty}^\infty \dim(\NMI_k(n))Y^k&=\sum_{k=1-n}^\infty \binom{2n+k-2}{n-1}Y^k\\
&=\sum_{l=0}^\infty \binom{n-1+l}{n-1}Y^{l-n+1}\\
&=\frac{Y^{1-n}}{(1-Y)^n}\\
&=\frac{Y}{(Y-Y^2)^n},
\end{align*}
and 
\begin{align*}
f_\NMI(X,Y)&=\sum_{n=1}^\infty \frac{YX^n}{(Y-Y^2)^n}=Y\sum_{n=1}^\infty \left(\frac{X}{Y-Y^2}\right)^n=Y\frac{\dfrac{X}{Y-Y^2}}{1-\dfrac{X}{Y-Y^2}}=\frac{XY}{Y-Y^2-X}. \qedhere
\end{align*}\end{proof}

\begin{example} 
This array gives the values of $\dim(\NMI_k(n))$ for small values of $k$ and $n$. 
\[\begin{array}{|c||c|c|c|c|c|c|c|c|c|c||c|}
\hline n\setminus k&-4&-3&-2&-1&0&1&2&3&4&5&\mbox{\scriptsize OEIS}\\
\hline\hline1&0&0&0&0&1&1&1&1&1&1&\\
\hline 2&0&0&0&1&2&3&4&5&6&7&\\
\hline 3&0&0&1&3&6&10&15&21&28&36&\mbox{\scriptsize A000217}\\
\hline 4&0&1&4&10&20&35&56&84&120&165&\mbox{\scriptsize A000292}\\
\hline 5&1&5&15&35&70&126&210&330&495&715&\mbox{\scriptsize A000332}\\
\hline\hline \mbox{\scriptsize OEIS}&\mbox{\scriptsize A002694}&\mbox{\scriptsize A002054}&\mbox{\scriptsize A001791}&\mbox{\scriptsize A088218}&\mbox{\scriptsize A000984}&\mbox{\scriptsize A001700}
&\mbox{\scriptsize A001791}&\mbox{\scriptsize A002054}&\mbox{\scriptsize A002694}&\mbox{\scriptsize A003516}&\\
\hline \end{array}\]
\end{example}

\begin{cor}\label{cor2.13}
The formal series of $\NMI_0$ is
\[f_{\NMI_0}(X)=\sum_{n=1}^\infty \binom{2n-2}{n-1}X^n=\frac{X}{\sqrt{1-4X}}.\]
\end{cor}

\subsection{Novikov algebras and $\NMI_0$}

Observe that 
\begin{align*}
\NMI_0(1)&=\vect(X_0),\\
\NMI_0(2)&=\vect(X_1X_0,X_0X_1),\\
\NMI_0(3)&=\vect(X_2X_0X_0,X_0X_2X_0,X_0X_0X_2,X_1X_1X_0,X_1X_0X_1,X_0X_1X_1).
\end{align*}
Here are examples of compositions in $\NMI_0$:
\begin{align*}
X_1X_0\circ (X_1X_0,X_0)&=X_2X_0X_0+X_1X_1X_0,&
X_1X_0\circ (X_0,X_1X_0)&=X_1X_1X_0.
\end{align*}
We obtain the relations
\begin{align*}
X_1X_0\circ (X_0,X_1X_0)^{(12)}&=X_1X_0\circ (X_0,X_1X_0),\\
X_1X_0\circ (X_1X_0,X_0)^{(23)}-X_1X_0\circ (X_0,X_1X_0)^{(23)}&=X_1X_0\circ (X_1X_0,X_0)-X_1X_0\circ (X_0,X_1X_0),
\end{align*}
both terms on the first row being equal to $X_1X_1X_0$ and both terms on the second row to $X_2X_0X_0$. 
As a consequence, $X_1X_0$ satisfies both the left non-assocative permutative and the right pre-Lie relations:

\begin{defi}
A Novikov algebra is a pair $(A,\triangleleft)$ where $A$ is a vector space and $\triangleleft:A\otimes A\longrightarrow A$ is a linear map such that:
\begin{itemize}
\item $(A,\triangleleft)$ is a right pre-Lie algebra:
\begin{align*}
&\forall x,y,z\in A, &(x\triangleleft y)\triangleleft z-x\triangleleft(y\triangleleft z)&=(x\triangleleft z)\triangleleft y-x\triangleleft(z\triangleleft y).
\end{align*}
\item $(A,\triangleleft)$ is a left non-associative permutative (shortly, NAP)  algebra:
\begin{align*}
&\forall x,y,z\in A, &x\triangleleft(y\triangleleft z)&=y\triangleleft(x\triangleleft z).
\end{align*}
\end{itemize}
The operad of Novikov algebras is denoted by $\nov$. It is generated by $\triangleleft\in \nov(2)$ and the relations
\begin{align*}
\triangleleft\circ (\triangleleft,I)-\triangleleft\circ(I,\triangleleft)&=\triangleleft\circ (\triangleleft,I)^{(23)}-\triangleleft\circ(I,\triangleleft)^{(23)},\\
\triangleleft\circ (I,\triangleleft)&=\triangleleft\circ (I,\triangleleft)^{(12)}.
\end{align*} 
\end{defi}

As Novikov algebras are pre-Lie algebras:

\begin{lemma}
There is a unique operad morphism $\theta:\prelie\longrightarrow\nov$, sending $\blacktriangleleft$ to $\triangleleft$. This morphism is surjective. 
\end{lemma}

Moreover, as $X_1X_0$ satisfies the right pre-Lie relation and the left NAP relation:

\begin{prop}
There exists a unique operad morphism $\theta_\nov:\nov\longrightarrow \NMI_0$, sending $\triangleleft$ to $X_1X_0$.
\end{prop}

\begin{remark}
Note that $X_1X_0=X_0X_0\circ (X_1,X_0)$. As a consequence, if $(A,m,D)$ is a differential commutative algebra, then it is a Novikov algebra, with 
\begin{align*}
&\forall a,b\in A,&a\triangleleft b&=D(a)b.
\end{align*}
\end{remark}

\subsection{From $\prelie$ to $\NMI_0$}

\begin{notation}
By composition, we obtain an operad morphism $\theta_\prelie=\theta_\nov\circ \theta:\prelie\longrightarrow \NMI_0$, sending $\blacktriangleleft$ to $X_1X_0$. 
\end{notation}

\begin{notation}
Let $F$ be an $n$-indexed rooted forest. For any $i\in [n]$, we denote by $f_F(i)$ the fertility of the vertex $i$ in $F$. In particular, if $i$ is a leaf of $F$, then $f_F(i)=0$. 
\end{notation}

\begin{prop}\label{prop2.17}
For any $n$-index rooted tree $T$,
\[\theta_\prelie(T)=X_{f_T(1)}\cdots X_{f_T(n)}.\]
\end{prop}

\begin{proof}
Let us firstly  prove that $\theta_\prelie(c_n)=X_{n-1}X_0^{n-1}$ by induction on $n$ (recall that $c_n$ is the corolla defined in Notation \ref{not1.1}). This is obvious if $n=1$ or $2$. 
Let us assume the result at rank $n$, with $n\geq 2$. 
\begin{align*}
\tddeux{$1$}{$2$}\circ_1 c_n&=c_{n+1}+\sum_{i=2}^n c_n \circ_i\tddeux{$1$}{$2$}^{(i+1,\ldots,n)}.
\end{align*}
Therefore,
\begin{align*}
\theta_\prelie(c_{n+1})&=X_1X_0\circ_1 X_{n-1}X_0^{n-1}-\sum_{i=2}^n X_{n-1}X_0^{n-1}\circ_i X_1X_0,X_0^{(i+1,\ldots,n)}\\
&=X_n X_0^n+\sum_{i=2}^n X_1X_0^{i-2}X_1X_0^{n-i+1}-\sum_{i=2}^n{X_{n-1}X_0^{i-2}X_1X_0^{n-i+1}}^{(i+1,\ldots,n)}\\
&=X_n X_0^n+\sum_{i=2}^n X_1X_0^{i-2}X_1X_0^{n-i+1}-\sum_{i=2}^nX_{n-1}X_0^{i-2}X_1X_0^{n-i+1}\\
&=X_nX_0^n.
\end{align*}

Let us prove the result in the general case by induction on $n$. If $n=1$, then $T=\tdun{$1$}$ and the result is obvious. 
Let us assume the result at all ranks $<n$, with $n\geq 2$. Let $T$ be a $n$-indexed rooted tree. 
We denote by $T_1,\ldots,T_k$ the rooted subtrees born from the root of $T$. There exists a permutation $\sigma\in \sym_n$ such that, in $T^\sigma=T'$:
\begin{itemize}
\item $1$ is the root of $T^\sigma$.
\item If $x\in V(T'_i)$ and $y\in V(T'_j)$, with $i<j$, then $x<y$. 
\end{itemize} 
For any $i\in [k]$, let $T'_i$ be the indexed tree obtained from $T_i$ by the unique increasing re-indexation. Then
\[T'=c_{k+1}\circ (\tdun{$1$},T'_1,\ldots,T'_k).\]
By the first step and the induction hypothesis applied to $T'_1,\ldots,T'_k$, 
\begin{align*}
\theta_\prelie(T')&=X_k X_0^k\circ (X_0, X_{f_{T'_1}(1)}\cdots X_{f_{T'_1}(n_1)},\ldots,X_{f_{T'_k}(1)}\cdots X_{f_{T'_k}(n_k)})\\
&=X_kX_{f_{T'_1}(1)}\cdots X_{f_{T'_1}(n_1)}\cdots X_{f_{T'_k}(1)}\cdots X_{f_{T'_k}(n_k)}\\
&=X_{f_{T'}(1)}\cdots X_{f_{T'}(n)}.
\end{align*}
Making $\sigma^{-1}$ acting on both side,
\begin{align*}
\theta_\prelie(T)&=X_{f_{T'}(\sigma(1))}\cdots X_{f_{T'}(\sigma(n))}=X_{f_T(1)}\cdots X_{f_T(n)}.
\end{align*}
So the result holds for any indexed tree $T$. 
\end{proof}

\begin{lemma}\label{lem2.18}
Let $(k_1,\ldots,k_n)$ be a sequence of integers such that $k_1+\cdots+k_n\leq n-1$. There exists an  $n$-indexed rooted forest $F$ such that $f_F(i)=k_i$ for any $i\in [n]$.
If moreover $k_1+\cdots+k_n=n-1$, then $F$ is a tree.
\end{lemma}

\begin{proof}
By induction on $n$. If $k_1=\cdots=k_n=0$ (which covers the case $n=1$), then one can take $F=\tdun{$1$}\ldots \tdun{$n$}$.
Let us assume that the result is true at any rank $<n$. Up to a permutation, we can assume that $k_1\leq \cdots \leq k_n$. We already cover the case where $k_n=0$.
Otherwise, 
\[k_1+\cdots+k_{n-1}\leq n-1-k_n \leq n-2.\]
By the induction hypothesis, there exists an $(n-1)$-indexed forest $F'$ such that $f_{F'}(i)=k_i$ for any $i\in [n]$. The number of vertices of $F'$ is $n-1$ and the number of edges of $F'$ is $k_1+\cdots+k_{n-1}$.   
Therefore, the number of roots of $F'$ is 
\[n-1-k_1-\cdots-k_{n-1}\geq k_n.\]
By adding a root indexed by $n$ to $k_n$ roots of $F'$, we obtain a forest $F$ such that $f_F(i)=k_i$ for any $i\in [n]$.\\

Let $F$ be an $n$-index rooted forest such that $f_F(i)=k_i$ for any $i\in [n]$. The number of edges of such an $F$ is $k_1+\cdots+k_n$, so the number of roots of $F$ is
\[n-k_1-\cdots-k_n.\]
Therefore, $F$ is a tree if, and only if, $k_1+\cdots+k_n=n-1$.
\end{proof}

\begin{prop}\label{prop2.19}
The operad morphism $\theta_\nov$, sending $\triangleleft$ to $X_1X_0$, is an isomorphism from $\nov$ to $\NMI_0$.
\end{prop}

\begin{proof}
Let $X_{k_1}\cdots X_{k_n}$ be a noncommutative multi-index of degree $0$. Therefore, $k_1+\cdots+k_n=n-1$. By Lemma \ref{lem2.18}, there exists an $n$-index rooted tree $T$
such that $f_T(i)=k_i$ for any $i\in [n]$. By Proposition \ref{prop2.17}, 
\[\theta_\prelie(T)=\theta_\nov(\theta(T))=X_{k_1}\ldots X_{k_n}.\]
So $\theta_\nov$ is surjective (and so is $\theta_\prelie$). By Corollary \ref{cor2.13} and \cite{Dzumi2002,Dzumi2014}, for any $n\geq 1$,
\[\dim(\nov(n))=\dim(\NMI_0(n))=\binom{2n-2}{n-1},\]
so $\theta_\nov$ is an isomorphism. 
\end{proof}

\section{Induced algebraic structures}

\subsection{Pre-Lie and brace structures}

From \cite[Proposition 37]{Foissy55}, $\displaystyle \bigoplus_{n\geq 1} \NMI(n)=\Alg$ inherits a brace structure, given by
\begin{align*}
&\forall p\in \NMI(n),\: \forall p_1,\ldots,p_k \in \NMI,&\{ p;p_1,\ldots,p_k\}&=\sum_{1\leq i_1<\ldots<i_k\leq n}p\circ_{i_1,\ldots,i_k}(p_1,\ldots,p_k),
\end{align*}
with the notation
\[p\circ_{i_1,\ldots,i_k}(p_1,\ldots,p_k)=p\circ (I,\ldots,I,p_1,I,\ldots,I,p_k,I,\ldots,I),\]
where $p_j$ is in position $i_j$ for any $j\in [k]$. This gives:

\begin{prop}\label{prop3.1}
The brace structure on $\NMI$ induced by the operadic structure is given by
\begin{align*}
&\forall X_{i_1}\cdots X_{i_n}\in \NMI(n),\: \forall P_1,\ldots,P_k\in \NMI,\\
\{X_{i_1}\cdots X_{i_n};P_1,\ldots,P_k \}&=\sum_{1\leq j_1<\ldots<j_k\leq n}X_{i_1}\cdots X_{i_{j_1}-1}D^{i_{j_1}}(P_1)\cdots D^{i_{j_k}}(P_k)X_{i_{j_k}+1}\cdots X_{i_n}.
\end{align*}
\end{prop}

This brace structure induces a pre-Lie product $\triangleleft$ on $\NMI$, given by $P\triangleleft Q=\{P;Q\}$.

\begin{cor}\label{cor3.2}
The pre-Lie product $\triangleleft$ induced by the operadic structure on $\NMI$ is given by
\begin{align*}
&\forall X_{i_1}\cdots X_{i_n}\in \NMI(n),\: \forall P\in \NMI,\\
X_{i_1}\cdots X_{i_n}\triangleleft P&=\sum_{1\leq k \leq n}X_{i_1}\cdots X_{i_k-1}D^{i_k}(P)X_{i_k+1}\cdots X_{i_n}.
\end{align*}
\end{cor}

This formula can be rewritten as follows:

\begin{cor}
For any $Q \in \Alg$, let us denote by $D_Q$ the derivation of $ \Alg$ which sends $X_i$ to $D^i(Q)$ for any $i\in \N$. 
The pre-Lie product $\triangleleft$ induced by the operad structures is given by
\begin{align*}
&\forall P,Q\in \Alg,&P\triangleleft Q&=D_Q(P). 
\end{align*}
\end{cor}

\begin{remark}
In particular, $D_{X_1}=D$, as $D^i(X_1)=X_{i+1}$ for any $i\in \N$. 
\end{remark}

This pre-Lie product $\triangleleft$ goes through the quotient $\coinv(\NMI)$, defined by
\[\coinv(\NMI)=\bigoplus_{n=1}^\infty \frac{\NMI(n)}{\vect(P-P^\sigma\mid P\in \NMI(n), \:\sigma \in \sym_n)}=\alg.\]
We obtain:

\begin{cor}
We give the non unitary polynomial algebra $\alg$ the derivation
\[D=\sum_{n=0}^\infty X_{n+1} \frac{\partial}{\partial X_n}.\]
For any $Q\in \alg$, we consider the derivation 
\[D_Q=\sum_{n=0}^\infty D^n(Q) \frac{\partial}{\partial X_n}.\]
Then the pre-Lie product $\triangleleft$ induced by the operadic structure on $\NMI$ is given by
\begin{align*}
&\forall P,Q\in \alg,&P\triangleleft Q&=D_Q(P). 
\end{align*}
\end{cor}

\begin{example}
For any $k\in \N$, for any $Q\in \alg$, $X_k\triangleleft Q=D^k(Q)$. In particular,
\begin{align*}
&\forall k,l\in \N,&X_k\triangleleft X_l&=X_{k+l}.
\end{align*}
We obtain a copy of the associative algebra $\K[X]$, seen as a pre-Lie algebra.
\begin{align*}
&\forall k,l\in \N,&X_0^{k+1}\triangleleft X_0^{l+1}&=(k+1) X_0^{k+l+1}.
\end{align*}
We obtain a copy of the Faà di Bruno pre-Lie algebra $\g_{\mathrm{FdB}(1)}$, with $e_k=X_0^{k+1}$.
\end{example}

Let us now describe the Guin-Oudom construction \cite{Oudom2005,Oudom2008} attached to this pre-Lie product.

\begin{prop}\label{prop3.5}
Let $P,P_1,\ldots,P_k\in \alg$. The Guin-Oudom construction attached to $\triangleleft$ gives
\[P\triangleleft P_1\cdots P_k=\sum_{l_1,\ldots,l_k\in \N} D^{l_1}(P_1)\cdots D^{l_k}(P_k)\frac{\partial^kP}{\partial X_{l_1}\cdots \partial X_{l_k}}.\]
\end{prop}

\begin{proof}
We first work in $\Alg$. There, the pre-Lie product comes from a brace structure $\{-;-\}$ (Proposition \ref{prop3.1}). For any noncommutative multi-index $P=X_{i_1}\cdots X_{i_n}$ and $P_1,\ldots,P_k\in \Alg$,
\begin{align*}
X_{i_1}\cdots X_{i_n}\triangleleft P_1\cdots P_k&=\sum_{\sigma \in \sym_k} \{X_{i_1}\cdots X_{i_n}; P_{\sigma(1)},\ldots, P_{\sigma(k)}\}\\
&=\sum_{\sigma \in \sym_k} \sum_{1\leq j_1<\cdots<j_k \leq n} X_{i_1}\cdots X_{i_{j_1}-1} D^{i_{j_1}}(P_{\sigma(1)})\cdots D^{i_{j_k}}(P_{\sigma(k)})X_{i_{j_k}+1}\cdots X_{i_n}.
\end{align*}
In the commutative quotient $\alg$, this gives
\begin{align*}
X_{i_1}\cdots X_{i_n}\triangleleft P_1\cdots P_k&=\sum_{\substack{1\leq j_1,\ldots,j_k \leq n,\\ \mbox{\scriptsize all distinct}}} D^{i_{j_1}}(P_1)\cdots D^{i_{j_k}}(P_k)\prod_{j\neq j_1,\ldots,j_k} X_{i_j}\\
&=\sum_{l_1,\ldots,l_k \in \N} D^{l_1}(P_1)\cdots D^{l_k}(P_k) \frac{\partial^k X_{i_1}\cdots X_{i_n}}{\partial X_{l_1}\cdots \partial X_{l_k}}.
\end{align*}
The result follows by linearity in $P$. 
\end{proof}

\subsection{Dual structure}

We identify the graded dual (for the graduation by the length) of $\Alg$ with itself, through the pairing defined by
\[ \langle X_{i_1}\cdots X_{i_m}, X_{j_1}\cdots X_{j_n}\rangle=\delta_{(i_1,\ldots,i_m),(j_1,\ldots,j_n)}.\]
The dual of the concatenation product is the (non counitary) deconcatenation coproduct $\deltadec$: for any noncommutative multi-index $X_{i_1}\cdots X_{i_n}$,
\begin{align*}
\deltadec(X_{i_1}\cdots X_{i_n})&=\sum_{k=1}^{n-1} X_{i_1}\cdots X_{i_k}\otimes X_{i_{k+1}}\cdots X_{i_n} .
\end{align*}

\begin{lemma}
Let $D':\Alg\longrightarrow \Alg$ be the unique derivation sending $X_0$ to $0$ and $X_i$ to $X_{i-1}$ for any $i\geq 1$. 
Then $D'$ is the dual of $D$:
\begin{align*}
&\forall P,Q\in \Alg,&\langle D(P),  Q\rangle&=\langle P,  D'(Q)\rangle.
\end{align*} \end{lemma}

\begin{proof}
As $D$ send $\NMI_k(n)$ to $\NMI_{k+1}(n)$ for any $(k,n)\in \Z\times \N$, there exists a unique map $D'$ such that  for any $P,Q\in \Alg$,
\[\langle D(P), Q\rangle=\langle P, D'(Q)\rangle.\]

Let us firstly prove that $D$ is a coderivation for the deconcatenation coproduct. Indeed, if $X_{i_1}\cdots X_{i_n}$ is a noncommutative multi-index,
\begin{align*}
\deltadec\circ D(X_{i_1}\cdots X_{i_n})&=\sum_{k=1}^n \deltadec(X_{i_1}\cdots X_{i_{k-1}}X_{i_k+1}X_{i_{k+1}}\cdots X_{i_n})\\
&=\sum_{k=1}^n \sum_{l=1}^{k-1} X_{i_1}\cdots X_{i_l} \otimes X_{i_{l+1}}\cdots X_{i_k+1}\cdots X_{i_n}\\
&+\sum_{k=1}^n \sum_{l=k}^{n-1} X_{i_1}\cdots X_{i_k+1}\cdots X_{i_l} \otimes X_{i_{l+1}}\cdots X_{i_n}\\
&=\sum_{l=1}^{n-1} \sum_{k=l+1}^{n-1} X_{i_1}\cdots X_{i_l} \otimes X_{i_{l+1}}\cdots X_{i_k+1}\cdots X_{i_n}\\
&+\sum_{l=1}^{n-1} \sum_{k=1}^l X_{i_1}\cdots X_{i_k+1}\cdots X_{i_l} \otimes X_{i_{l+1}}\cdots X_{i_n}\\
&=\sum_{l=1}^{n-1}X_{i_1}\cdots X_{i_l} \otimes D(X_{i_{l+1}}\cdots  X_{i_n})\\
&+\sum_{l=1}^{n-1} D(X_{i_1}\cdots X_{i_l}) \otimes X_{i_{l+1}}\cdots X_{i_n}\\
&=(\id_{\Alg} \otimes D+D\otimes \id_{\Alg})\circ \deltadec(X_{i_1}\cdots X_{i_n}),
\end{align*}
so $D$ is a coderivation for $\deltadec$. Dually, $D'$ is a derivation for the concatenation product. Therefore, it is enough to compute its value on $X_i$, for $i\in \N$. 
If $P$ is a noncommutative multi-index of length $n\geq 2$, then $D(P)$ is a sum of noncommutative multi-indices of the same length $n$. We deduce that 
\[\langle P, D'(X_i)\rangle=\langle D(P) , X_i\rangle=0.\]
So there exists a sequence $(a_{i,j})_{i,j\in \N}$ of scalars such that 
\begin{align*}
&\forall j\in \N,&D'(X_j)&=\sum_{i=0}^\infty a_{i,j}X_i. 
\end{align*}
Moreover, for any $i,j\in \N$,
\begin{align*}
a_{i,j}&=\langle X_i , D'(X_j)\rangle=\langle D(X_i), X_j \rangle=\langle X_{i+1}, X_j\rangle=\delta_{i+1,j}. 
\end{align*}  
Consequently, if $j=0$, for any $i\in \N$, $a_{i,j}=0$, and $D'(X_0)=0$. If $j\geq 1$, $a_{i,j}=\delta_{i,j-1}$ for any $i\in \N$, so $D'(X_j)=X_{j-1}$. 
\end{proof}

The operadic composition can be dualized into a coproduct on $T(\Alg)$, see \cite[Proposition 47]{Foissy55}. In order to avoid confusions with the concatenation product of $\Alg$, the concatenation product of $T(\Alg)$ is denoted by $\mid$.
For example, if $i\neq j$, $X_iX_j$, $X_jX_i$, $X_i\mid X_j$ and $X_j\mid X_i$ are pairwise distinct elements of $T(\Alg)$. 

\begin{prop}\label{prop3.7}
The operadic structure on $\Alg$ dually induces the coproduct $\deltaNMI$ on $T(\Alg)$, multiplicative for $\mid$, given on any $P\in \Alg$ by
\begin{align}
\label{EQ4}
\deltaNMI(P)&=\sum_{k=1}^\infty \sum_{(j_1,\ldots,j_k)\in \N^k} X_{j_1}\cdots X_{j_k}\otimes \left(\mid^{(k-1)}\circ (D'^{j_1}\otimes\cdots \otimes D'^{j_k})\circ \deltadec^{(k-1)}(P)\right).
\end{align}
The counit is the unique algebra map $\epsilon_{\deltaNMI}:T(\Alg)\longrightarrow \K$ such that for any noncommutative multi-index $P$, $\epsilon_{\deltaNMI}(P)=\delta_{P,X_0}$. Equivalently,
\begin{align*}
&\forall P\in \Alg,&\epsilon_{\deltaNMI}(P)&=\frac{\partial P}{\partial X_0}(0,0,\ldots).
\end{align*}
This coproduct $\deltaNMI$ is homogeneous of degree $0$ for the weight and for the degree.
\end{prop}

\begin{remark}
The derivation $D'$ is locally nilpotent. Therefore, the sum appearing in (\ref{EQ4}) is actually finite.
\end{remark}

\begin{proof}
Let $P\in \Alg$. By duality with the operadic composition, the coproduct has the form
\begin{align*}
\deltaNMI(P)&=\sum_{k=1}^\infty \sum_{(j_1,\ldots,j_k)\in \N^k}X_{j_1}\cdots X_{j_k} \otimes \delta_{j_1,\ldots,j_k}(P),
\end{align*}
with $\delta_{j_1,\ldots,j_k}(P)\in \Alg^{\mid k}$. Moreover, for any $P_1,\ldots,P_k\in \Alg$,
\begin{align*}
\langle \delta_{j_1,\ldots,j_k}(P), P_1\mid\cdots \mid P_k\rangle&=\langle \deltaNMI(P), X_{j_1}\cdots X_{j_k}\otimes P_1\mid\cdots \mid P_k\rangle\\
&=\langle P, X_{j_1}\cdots X_{j_k}\circ (P_1,\ldots,P_k)\rangle\\
&=\langle P, D^{j_1}(P_1)\cdots D^{j_k}(P_k)\rangle\\
&=\langle \deltadec^{(k-1)}(P), D^{j_1}(P_1)\otimes\cdots \otimes D_{j_k}(P_k)\rangle\\
&=\langle (D'^{j_1} \otimes\cdots \otimes D'^{j_k})\circ \deltadec^{(k-1)}(P), P_1\otimes\cdots \otimes P_k\rangle\\
&=\langle \mid^{(k-1)}\circ (D'^{j_1}\otimes\cdots \otimes D'^{j_k})\circ \deltadec^{(k-1)}(P), P_1\mid \cdots \mid P_k\rangle,
\end{align*}
which gives the formula for $\deltaNMI$. For any noncommutative multi-index $P$, as $X_0$ is the unit of the operad $\NMI$, $\epsilon_{\deltaNMI}(P)=\langle P,X_0\rangle=\delta_{P,X_0}$.\\

The operadic composition $\circ$ is homogeneous of degree $0$ for the degree (Proposition \ref{prop2.10}). Dually, $\deltaNMI$ is homogeneous of degree $0$ for the degree. 
Moreover, $\deltadec$ is homogeneous of degree $0$ for the weight, and $D'$ is homogeneous of degree $-1$ for the weight. 
Therefore, if $P$ is homogeneous of degree $n$ for the weight, then $\mid^{(k-1)}\circ (D'^{j_1}\otimes\cdots \otimes D'^{j_k})\circ \deltadec^{(k-1)}(P)$ is homogeneous of weight $n-j_1-\cdots-j_k$.
As $X_{j_1}\cdots X_{j_k}$ is homogeneous of weight $k$, $\deltaNMI(P)$ is homogeneous of weight $n$.
\end{proof}

\begin{example}
Let $m,n,p\in \N$. 
\begin{align*}
\deltaNMI(X_m)&=\sum_{i=0}^{m}X_i\otimes X_{m-i},\\
\deltaNMI(X_mX_n)&=\sum_{i=0}^m\sum_{j=0}^n \frac{(i+j)!}{i!j!}X_{i+j}\otimes X_{m-i}X_{n-j}+\sum_{i=0}^m\sum_{j=0}^n X_iX_j \otimes X_{m-i}\mid X_{n-j},
\end{align*}
\begin{align*}
\deltaNMI(X_mX_nX_p)&=\sum_{i=0}^m \sum_{j=0}^n \sum_{k=0}^p\frac{(i+j+k)!}{i!j!k!}X_{i+j+k}\otimes X_{m-i}X_{n-j}X_{p-k}\\
&+\sum_{i=0}^m \sum_{j=0}^n \sum_{k=0}^p\frac{(i+j)!}{i!j!}X_{i+j}X_k\otimes X_{m-i}X_{n-j}\mid X_{p-k}\\
&+\sum_{i=0}^m \sum_{j=0}^n \sum_{k=0}^p\frac{(j+k)!}{j!k!}X_iX_{j+k}\otimes X_{m-i}\mid X_{n-j}X_{p-k}\\
&+\sum_{i=0}^m \sum_{j=0}^n \sum_{k=0}^pX_iX_jX_k\otimes X_{m-i}\mid X_{n-j}\mid X_{p-k}.
\end{align*}
\end{example}

Formula (\ref{EQ4}) can be developed when $P$ is a noncommutative multi-index:

\begin{cor}
For any noncommutative multi-index $X_{j_1}\cdots X_{j_n}$,
\begin{align*}
\deltaNMI(X_{j_1}\cdots X_{j_n})&=\sum_{i_1=0}^{j_1}\ldots \sum_{i_n=0}^{j_n} \sum_{1\leq p_1<\ldots <p_l=n}\frac{(i_1+\cdots+i_{p_1})!}{i_1!\ldots i_{p_1}!}\ldots \frac{(i_{p_{l-1}+1}+\cdots i_{p_l})!}{i_{p_{l-1}+1}!\ldots i_{p_l}!}\\
&X_{i_1+\cdots+i_{p_1}}\cdots X_{i_{p_{l-1}+1}+\cdots+i_{p_l}}\otimes X_{j_1-i_1}\cdots X_{j_{p_1}-i_{p_1}}\mid\cdots \mid X_{j_{p_{l-1}+1}-i_{p_{l-1}+1}}\cdots X_{j_{p_l}-i_{p_l}}.
\end{align*}
\end{cor}

The same formula holds in $S(\Alg)$, abelianization of $T(\Alg)$, with $\mid$ meaning there the usual product of the symmetric algebra $S(\Alg)$. 
Moreover, as $\coinv(\Alg)$ is a quotient pre-Lie algebra of $\Alg$, dually, the Guin-Oudom construction implies that $S(\inv(\Alg))$ is a Hopf subalgebra of $S(\Alg)$. 
As a vector space, we identify $\inv(\Alg)$ with $\algdual$ (note that we change the notation for the indeterminates, in order to avoid the confusions), trough the linear map
\[\theta:\left\{\begin{array}{rcl}
\algdual&\longrightarrow&\inv(\Alg)\\
x_0^{\alpha_0}\ldots x_n^{\alpha_n}&\longmapsto& \displaystyle\sum_{\sigma \in \sym_a} X_{i_{\sigma(1)}}\cdots X_{i_{\sigma(a)}},
\end{array}
\right.\]
with $a=\alpha_0+\cdots+\alpha_n$ and 
\[(i_1,\ldots,i_a)=(\underbrace{0,\ldots,0}_{\alpha_0},\ldots,\underbrace{n,\ldots,n}_{\alpha_n}).\]
We denote by $\Deltash$ the usual coproduct of $\K[x_i\mid i\in \N]$, defined as the multiplicative (for the usual product) coproduct such that $\Deltash(x_i)=x_i\otimes 1+1\otimes x_i$ for any $i\in \N$
and by $\deltash$ its reduced part, defined from $\algdual$ to $\algdual^{\otimes 2}$, obtained by deleting the primitive part:
\begin{align*}
&\forall P\in \algdual,&\deltash(P)&=\Deltash(P)-P\otimes 1-1\otimes P. 
\end{align*}
In other words, for any monomial $x_{i_1}\cdots x_{i_n}$,
\[\deltash(x_{i_1}\cdots x_{i_n})=\sum_{\emptyset \subsetneq I\subsetneq [n]} \prod_{j\in I}x_{i_j}\otimes \prod_{j\notin I} x_{i_j}.\]
Note that $\deltadec\circ \theta=(\theta \otimes \theta)\circ \deltash$. Through the identification $\theta$, we obtain from Proposition \ref{prop3.7} what follows:

\begin{prop}\label{prop3.9}
The operad structure on $\NMI$ induces the coproduct $\deltaNMI$ on $S(\alg)$, multiplicative for the usual product $\mid$, given on any $P\in \algdual$ by
\begin{align*}
\deltaNMI(P)&=\sum_{k=1}^\infty \sum_{(j_1,\ldots,j_k)\in \N^k} x_{j_1}\cdots x_{j_k}\otimes \left(\mid^{(k-1)}\circ (D'^{j_1}\otimes\cdots \otimes D'^{j_k})\circ \deltash^{(k-1)}(P)\right).
\end{align*}
Moreover, this coproduct $\deltaNMI$ is homogeneous of degree $0$ for the weight and for the degree.
\end{prop}

As $\tdelta$ is cocommutative and $\mid$ is commutative, we can obtain a more compact formula. We shall use for this the following notations:

\begin{defi}
We denote by $\Lambda$ is the set of nonzero sequences $\alpha=(\alpha_i)_{i\in \N}$ with a finite support. For any $\alpha \in \Lambda$, we put
\[x^\alpha=\prod_{i\in \N}x_i^{\alpha_i}\in \algdual.\]
For any $\alpha \in\Lambda$, we put
\begin{align*}
\ell(\alpha)=\ell(x^\alpha)&=\sum_{i=0}^\infty \alpha_i,&\omega(\alpha)=\omega(x^\alpha)&=\sum_{i=0}^\infty i\alpha_i,\\
\alpha!&=\prod_{i=0}^\infty \alpha_i!,&\deg(\alpha)=\deg(x^\alpha)&=\omega(x^\alpha)-\ell(x^\alpha)+1=1+\sum_{i=0}^\infty (i-1)\alpha_i.
\end{align*}\end{defi}

\begin{prop}\label{prop3.11}
For any $P\in  \algdual$,
\begin{align*}
\deltaNMI(P)&=\sum_{\alpha \in \Lambda} \frac{1}{\alpha!}x^\alpha \otimes 
\mid^{(\ell(\alpha)-1)}\circ\left(\bigotimes_{n=0}^\infty {D'^n}^{\otimes \alpha_n} \right)\circ \deltash^{(\ell(\alpha)-1)}(P).
\end{align*}\end{prop}

\subsection{Coproduct induced by the operadic morphism $\theta_\prelie$}

As explained in \cite[Corollary 51]{Foissy55}, the morphism $\theta_\prelie:\prelie\longrightarrow \NMI$, sending $\tddeux{$1$}{$2$}$ to $X_1X_0$, induces a coproduct $\DeltaNMI$ on $S(\algdual)$, making it a double bialgebra. 
Let us first describe the pre-Lie product $\blacktriangleleft$ on the free $\NMI$-algebra on one generator, which we identify with $\alg$, induced by $\theta_\prelie$:  for any $P,Q\in \Alg$, denoting  the canonical surjection by 
\[\pi:\Alg\longrightarrow \alg,\]
we obtain
\[\pi(P)\blacktriangleleft \pi(Q)=\pi(X_1X_0\circ(P,Q))=\pi(D(P)Q).\]
As a consequence:

\begin{prop}
The pre-Lie product $\blacktriangleleft$ induced by $\theta_\prelie$ on $\alg$ is given by
\begin{align*}
&\forall P,Q\in \alg,&P\blacktriangleleft Q&=D(P)Q.
\end{align*}
\end{prop}

\begin{remark} \label{rk3.3}
This is an example of a more general construction: 
if $(A,m,D)$ is a commutative differential algebra, not necessarily unitary, then it is a pre-Lie algebra, with
\begin{align*}
&\forall a,b\in A,&a\blacktriangleleft b&=D(a)b.
\end{align*}
\end{remark}

Let us now describe the Guin-Oudom construction applied to $(\alg,\blacktriangleleft)$. The product of $S(\alg)$ is still denoted by $\mid$.

\begin{prop}
The Guin-Oudom extension of $\blacktriangleleft$ to $S(\alg)$ is given by
\begin{align*}
&\forall P,P_1,\ldots,P_k \in \alg,&P\blacktriangleleft P_1\mid\cdots \mid P_k&=D^k(P)P_1\cdots P_k.
\end{align*}
\end{prop}

\begin{proof}
Let us choose $Q,Q_1,\ldots,Q_k\in \Alg$ such that $\pi(Q)=P$ and $\pi(Q_i)=P_i$ for any $i\in [k]$. Then
\begin{align*}
P\blacktriangleleft P_1\mid\cdots \mid P_k&=\pi(\phi_{\prelie}(c_{k+1})\circ (Q,Q_1,\ldots,Q_k))\\
&=\pi(X_kX_0^k\circ (Q,Q_1,\ldots,Q_k))\\
&=\pi(D^k(Q)Q_1\cdots Q_k)\\
&=D^k(P)P_1\cdots P_k, 
\end{align*}
where $c_{k+1}$ is the corolla of Notation \ref{not1.1}.
\end{proof}

\begin{remark}
This can be extended to the more general case of a differential algebra $(A,m,D)$,
as defined in Remark \ref{rk3.3}:
\begin{align*}
&\forall a,a_1,\ldots,a_n\in A,&a\blacktriangleleft a_1\mid\cdots\mid a_n&=D^n(a)a_1\cdots a_n.
\end{align*}
\end{remark}

Dually, the Grossman-Larson product $*$ associated to $\blacktriangleleft$ can be dualized into a coproduct $\DeltaNMI$ on $S(\algdual)$, making it a Hopf algebra with the product $\mid$. 
Here, we extend the pairing between $\algdual$ and $\alg$ to $S(\algdual)$ and $S(\alg)$, such that it is a Hopf pairing between $(S(\algdual),\mid,\Deltash)$ and $(S(\alg),\mid,\Deltash)$, 
where $\Deltash$ is for the usual deshuffling coproduct of these symmetric algebras. 

\begin{prop}\label{prop3.14}
The coproduct $\DeltaNMI$ is given on any $P\in \algdual$ by
\begin{align}
\label{EQ5}
\DeltaNMI(P)&=P\otimes 1+1\otimes P+\sum_{k=1}^\infty \frac{1}{k!}\left(D'^k\otimes \mid^{(k-1)}\right)\circ \deltash^{(k)}(P).
\end{align}
This coproduct is homogeneous of degree $0$ for the length and the degree.
\end{prop}

\begin{remark}
As both $\deltash$ and $D'$ are locally nilpotent, the sum in (\ref{EQ5}) is is actually finite.
\end{remark}

\begin{proof}
As the dual of a Grossman-Larson product, $\DeltaNMI$ has the form
\[\DeltaNMI(P)=P\otimes 1+1\otimes P+\sum_{\alpha \in \Lambda} x^\alpha \otimes \Delta_\alpha(P),\]
where $\Delta_\alpha(P)\in S(\alg)$ for any $\alpha\in \Lambda$. By duality, for any $Q,Q_1,\ldots,Q_k \in \alg$,
\begin{align*}
\langle \DeltaNMI(P),Q\otimes Q_1\mid\cdots \mid Q_k\rangle&=\langle P,Q*Q_1\mid\cdots \mid Q_k\rangle\\
&=\langle P,Q\blacktriangleleft Q_1\mid\cdots \mid Q_k\rangle+0\\
&=\langle P,D^k(Q)Q_1\ldots Q_k\rangle\\
&=\langle \deltash^{(k)}(P),D^k(Q)\otimes Q_1\otimes\cdots \otimes Q_k\rangle\\
&=\langle\left (D'^k\otimes \id^{\otimes k}\right) \circ \deltash^{(k)}(P),Q\otimes Q_1\otimes\cdots \otimes Q_k\rangle\\
&=\frac{1}{k!}\sum_{\sigma \in \sym_k}\langle\left (D'^k\otimes \id^{\otimes k}\right) \circ \deltash^{(k)}(P),Q\otimes Q_{\sigma(1)}\otimes\cdots \otimes Q_{\sigma(k)}\rangle\\
&=\frac{1}{k!} \langle\left (D'^k\otimes \id^{\otimes k}\right) \circ \deltash^{(k)}(P),Q\otimes \tdelta^{(k-1)}(Q_1\mid\cdots \mid Q_k)\rangle\\
&=\frac{1}{k!} \langle \left(\id \otimes \mid^{(k-1)}\right)\circ\left (D'^k\otimes \id^{\otimes k}\right) \circ \deltash^{(k)}(P),Q\otimes Q_1\mid\cdots \mid Q_k\rangle.
\end{align*}
We used the cocommutativity of $\deltash$ for the sixth equality. From the form of the coproduct $\DeltaNMI$, we deduce the announced formula for $\DeltaNMI(P)$.\\

It remains to prove the result for the homogeneity. Firstly, the deshuffling coproduct $\deltash$ is homogeneous of degree $0$ for the length and the weight.
The derivation $D'$ is homogeneous of degree $0$ for the length and of degree $-1$ for the weight. By composition, $\blacktriangleleft$ is homogeneous of degree $0$ for the length, except in degree $0$. 

Let $P\in \algdual$, of length $n$ and weight $m$ (so of degree $m-n+1$). 
For any $k\geq 1$, $\deltash^{(k)}(P)$ is a linear span of tensors of monomials $P_1\otimes\cdots \otimes P_{k+1}$, with
\begin{align*}
\ell(P_1)+\cdots+\ell(P_{k+1})&=n,&\omega(P_1)+\cdots+\omega(P_{k+1})&=m.
\end{align*} 
Moreover,
\begin{align*}
\ell\left(D'^k(P_1)\right)&=\ell(P_1),&\omega\left(D'^k(P_1)\right)&=\omega(P_1)-k.
\end{align*}
We obtain
\begin{align*}
\ell\left(D'^k(P_1)\otimes P_2\mid \cdots \mid P_{k+1}\right)&=\ell(P_1)+\cdots+\ell(P_{k+1})=n,\\
\omega\left(D'^k(P_1)\otimes P_2\mid \cdots \mid P_{k+1})\right)&=\omega(P_1)+\cdots+\omega(P_{k+1})-k=m-k,\\
\deg\left(D'^k(P_1)\otimes P_2\mid \cdots \mid P_{k+1}\right)&=\deg\left(D^k(P_1)\right)+\deg(P_2)+\cdots+\deg(P_{k+1})\\
&=\omega(P_1)+\cdots+\omega(P_{k+1})-\ell(P_1)-\cdots-\ell(P_{k+1})+1+k\\
&=m-n+1\\
&=\deg(P),
\end{align*}
so $\blacktriangleleft$ is homogeneous for the degree.  
\end{proof}

\begin{example}
Let $i,j,k\in \N$. We obtain
\begin{align*}
\DeltaNMI(x_i)&=x_i\otimes 1+1\otimes x_i,\\
\DeltaNMI(x_ix_j)&=x_ix_j\otimes 1+1\otimes x_ix_j+x_{i-1}\otimes x_j+x_{j-1}\otimes x_i,\\
\DeltaNMI(x_ix_jx_k)&=x_ix_jx_k\otimes 1+1\otimes x_ix_jx_k+x_{i-1}\otimes x_jx_k+x_{j-1}\otimes x_ix_k+x_{k-1}\otimes x_ix_j\\
&+x_{j-1}x_k\otimes x_i+x_jx_{k-1}\otimes x_i+x_{i-1}x_k\otimes x_j+x_ix_{k-1}\otimes x_j+x_{i-1}x_j\otimes x_k\\
&+x_ix_{j-1}\otimes x_k+x_{i-2}\otimes x_j\mid x_k+x_{j-2}\otimes x_i\mid x_k+x_{k-2}\otimes x_i\mid x_j.
\end{align*}\end{example}

%

From \cite[Corollary 51]{Foissy55}:

\begin{theo}
$(S(\algdual),\mid,\DeltaNMI,\deltaNMI)$ is a double bialgebra.
\end{theo}

\subsection{A link with the Connes-Kreimer Hopf algebra}

By functoriality, the operad morphism $\theta_\prelie:\prelie\longrightarrow \NMI$ induces a double bialgebra morphism $\Psi:S(\algdual)\longrightarrow \HCK$. Let us describe this morphism. 

\begin{notation}\begin{enumerate}
\item For any $\alpha \in \Lambda$, we put  $\displaystyle x^\alpha=\prod_{i=0}^\infty x_i^{\alpha_i}$.
The pairing between $\algdual$ and $\alg$ is given by
\begin{align*}
&\forall \alpha,\beta\in \Lambda,&\langle x^\alpha,X^\beta\rangle&=\alpha!\delta_{\alpha,\beta}.
\end{align*}
\item For any rooted tree $T\in \calT$, let us denote by $M(T)$ the monomial 
\[M(T)=\prod_{v\in V(T)}x_{f_T(v)}\in \algdual.\] 
For example, 
\begin{align*}
\begin{array}{|c||c|c|c|c|c|c|c|}
\hline T&\tun&\tdeux&\ttroisun&\ttroisdeux&\tquatreun&\tquatredeux,\tquatrequatre&\tquatrecinq\\
\hline\hline&&&&&&& \\[-3mm]
M(T)&x_0&x_1x_0&x_2x_0^2&x_1^2x_0&x_3x_0^3&x_2x_1x_0^2&x_1^3x_0\\
\hline \end{array}\end{align*}\end{enumerate}\end{notation}

\begin{theo}\label{theo3.16}
For any monomial $x^\alpha \in \algdual$, 
\[\Psi(x^\alpha)=\alpha!\left(\sum_{T\in \calT,\: M(T)=x^\alpha}\frac{1}{s_T}T\right).\]
In particular, if $\deg(x^\alpha)\neq 0$, then $\Psi(x^\alpha)=0$.
\end{theo}

\begin{proof}
By duality, $\Psi(x^\alpha)$ is of the form
\[\Psi(x^\alpha)=\sum_{T\in \calT} a_{\alpha,T}T,\]
where the $a_{\alpha,T}$ are scalars. For any rooted tree $T\in \calT$,
\begin{align*}
a_{\alpha,T}s_T&=\langle \Psi(x^\alpha),T\rangle=\langle x^\alpha,\phi_{\prelie}(T)\rangle
=\langle x^\alpha,\prod_{v\in V(T)} X_{f_T(v)}\rangle=\alpha! \delta_{x^\alpha,M(T)},
\end{align*}
which gives the formula for $\Psi(x^\alpha)$.\\

Let us assume that $\Psi(x^\alpha)\neq 0$. We put $x^\alpha=x_{i_1}\ldots x_{i_n}$ . There exists an $n$-indexed rooted tree $T$ such that for any $j\in [N]$, $f_T(j)=i_j$. Consequently, the number of edges of $T$ is
\[n-1=\sum_{j=1}^n f_T(j)=i_1+\cdots+i_n.\]
Therefore, $\deg(x^\alpha)=0$. Equivalently, if $\deg(x^\alpha)\neq 0$, then $\Psi(x^\alpha)=0$. 
\end{proof}

\begin{example}\begin{enumerate}
\item For any $n\geq 1$,
\begin{align*}
\Psi(x_{n-1}x_0^{n-1})&=\frac{(n-1)!}{(n-1)!}C_n=C_n,&\Psi(x_1^{n-1}x_0)&=\frac{(n-1)!}{1}L_n=(n-1)!L_n. 
\end{align*}
\item Here are more examples:
\begin{align*}
\Psi(x_0)&=\tun,&
\Psi(x_1x_0)&=\tdeux,\\
\Psi(x_2x_0^2)&=\ttroisun,&
\Psi(x_1^2x_0)&=2\ttroisdeux,\\
\Psi(x_1^3x_0)&=\tquatreun,&
\Psi(x_2x_1x_0^2)&=2\tquatredeux+\tquatrequatre,\\
\Psi(x_1^3x_0)&=6\tquatrecinq,&
\Psi(x_4x_0^3)&=\tcinqun,\\
\Psi(x_3x_1x_0^3)&=3\tcinqdeux+\tcinqdix,&
\Psi(x_2^2x_0^3)&=6\tcinqsix,\\
\Psi(x_2x_1^2x_0^2)&=4\tcinqhuit+2\tcinqcinq+4\tcinqonze+2\tcinqtreize,&
\Psi(x_1^4x_0)&=24\:\tcinqquatorze.
\end{align*}\end{enumerate}\end{example}

The coefficients appearing in $\Psi(x^\alpha)$ are natural integers:

\begin{lemma}
Let $T$ be a rooted tree. We put $M(T)=x^\alpha$. Then $s_T$ divides $\alpha!$.
\end{lemma}

\begin{proof}
We denote by $n$ the cardinality of $V(T)$. Let $G_T$ be the group of automorphisms of $T$. For any $k\in \N$, we denote by $V_k(T)$ the set of vertices of $T$ of fertility $k$. 
Note that if $k\geq n$, then $V_k(T)$ is empty.  Then for any $k\in  \N$, $G$ stabilizes $V_k(T)$, so $G_T$ can be seen as a subgroup of
\[H=\prod_{k=0}^{n-1} \mathfrak{S}_{V_k(T)}.\]
By Lagrange's theorem, $|G_T|=s_T$ divides $|H|=\alpha!$.  
\end{proof}

\begin{notation}
For any tree $T$, we denote by $p_T$ the number of embeddings of $T$ in the plane. 
\begin{align*}
\begin{array}{|c||c|c|c|c|c|c|c|c|c|c|c|c|}
\hline T&\tun&\tdeux&\ttroisun&\ttroisdeux&\tquatreun&\tquatredeux=\tquatretrois&\tquatrequatre&\tquatrecinq&\tcinqdeux=\tcinqtrois=\tcinqquatre&\tcinqsix=\tcinqsept&\tcinqhuit\\
\hline\hline&&&&&&&&&&& \\[-4mm]
p_T&1&1&1&1&1&2&1&1&3&2&2\\
\hline \end{array}\end{align*}
\end{notation}

\begin{prop}\label{prop3.18}
For any $\alpha\in \Lambda$, we put
\[c_\alpha=c_{x^\alpha}=\frac{\displaystyle \prod_{i=0}^\infty \alpha_i!}{\displaystyle \prod_{i=0}^\infty i!^{\alpha_i}}.\]
Then
\[\Psi(x^\alpha)=c_\alpha\left(\sum_{T\in \calT,\: M(T)=x^\alpha} p_TT\right).\]
\end{prop}

\begin{proof}
We have to prove that for any $T\in \calT$, 
\[c_{M(T)}p_T=\frac{M(T)!}{s_T}.\]
We put, for any tree $T\in \calT$, 
\[a_T=\frac{c_{M(T)}p_Ts_T}{M(T)!}.\]
Note that for any rooted tree $T$, putting $M(T)=x^\alpha$,
\begin{align*}
\frac{M(T)!}{c_{M(T)}}&=\prod_{i=0}^\infty i!^{\alpha_i}=\prod_{v\in V(T)} f_T(v)!
\end{align*}
Let us prove that $aT=1$ by induction on the number of vertices of $T$. This is obvious if $T=\tun$. 
Otherwise, let us put $T=B^+(T_1^{\beta_1}\ldots T_k^{\beta_k})$, where the $T_i$'s are pairwise distinct trees. Then
\begin{align*}
p_T&=p_{T_1}^{\beta_1}\ldots p_{T_k}^{\beta_k}\frac{(\beta_1+\cdots+\beta_k)!}{\beta_1!\ldots \beta_k!},\\
s_T&=s_{T_1}^{\beta_1}\ldots s_{T_k}^{\beta_k}\beta_1!\ldots \beta_k!,\\
\frac{M(T)!}{c_{M(T)}}&=(\beta_1+\cdots+\beta_k)! \prod_{j=i}^k \left(\frac{M(T_i)!}{c_{M(T_i)}}\right)^{\beta_i}.
\end{align*}
Combining, we obtain
\[a_T=a_{T_1}^{\beta_1}\ldots a_{T_k}^{\beta_k}.\]
By the induction hypothesis, $a_T=1$. 
\end{proof}

\begin{example}
\[\begin{array}{|c||c|c|c|c|c|c|c|c|c|c|c|c|}
\hline &&&&&&&&&&&&\\[-3mm] \alpha&x_0&x_1x_0&x_2x_0^2&x_1^2x_0&x_3x_0^3&x_2x_1x_0^2&x_1^3x_0&x_4x_0^3&x_3x_1x_0^3&x_2^2x_0^3&x_2x_1^2x_0^2&x_1^4x_0\\
\hline\hline c_\alpha&1&1&1&2&1&1&6&1&1&3&2&24\\
\hline
\end{array}
\]
The coefficients $c_{x^\alpha}$ are not necessarily integers. For example,
\begin{align*}
c_{x_4x_2x_0^5}&=\frac{5}{2},&c_{x_3^2x_0^5}&=\frac{20}{3}.
\end{align*}\end{example}

\subsection{Generation by a Dyson-Schwinger equation}

\begin{prop}\label{prop3.19}
Let $f=\displaystyle \sum_{k=0}^\infty a_kh^k\in \mathbb{C}[[h]]$ be a formal series. We consider the fixed-point equation
\begin{align}
\label{EQ6}
\scrT=B^+\left(\sum_{i=0}^\infty a_i X_i \scrT^i\right),
\end{align}
where $\scrT\in \HCK[[X_i\mid i\in \N]]$. This equation has a unique solution, of the form
\[\scrT=\sum_{\alpha \in \Lambda} T_\alpha X^\alpha.\]
The subalgebra of $\HCK$ generated the elements $T_\alpha$, with $\alpha\in \Lambda$, is denoted by $H_f$.
\end{prop}

\begin{proof}
Let $\scrT=\displaystyle \sum_{\alpha \in \Lambda} T_\alpha X^\alpha \in \HCK[[X_i\mid i\in \N]]$. We rewrite it under the form
\[\scrT=\sum_{T\in \calT} P_T(X)T,\]
where $T$ runs over the set of rooted trees and $P_T\in \K[[X_i\mid i\in \N]]$ for any tree $T$. 
Then
\begin{align*}
&\scrT=B^+\left(\sum_{i=0}^\infty a_i X_i \scrT^i\right)\\
&\Longleftrightarrow\scrT
=\sum_{\substack{T_1,\ldots, T_k\\ \mbox{\scriptsize pairwise distinct trees},\\ \alpha=(\alpha_1,\ldots,\alpha_k)\in \N^k}}
a_{\ell(\alpha)}X_{\ell(\alpha)}\frac{\ell(\alpha)!}{\alpha!}P_{T_1}(X)^{\alpha_1}\ldots P_{T_k}(X)^{\alpha_k}B^+(T_1^{\alpha_1}\ldots T_k^{\alpha_k})\\
&\Longleftrightarrow \mbox{for any $T_1,\ldots,T_k$ pairwise distinct rooted trees, and $\alpha=(\alpha_1,\ldots,\alpha_k)\in \N^k$},\\
&P_{B^+(T_1^{\alpha_1}\ldots T_k^{\alpha_k})}(X)=a_{\ell(\alpha)}X_{\ell(\alpha)}\frac{\ell(\alpha)!}{\alpha!}P_{T_1}(X)^{\alpha_1}\ldots P_{T_k}(X)^{\alpha_k}.
\end{align*}
This defines a unique sequence $(P_T(X))_{T\in \calT}$. Therefore, the Dyson-Schwinger equation (\ref{EQ6}) has a unique solution. 
\end{proof}

\begin{example}
We obtain
\begin{align*}
\scrT&=a_0\tun X_0+a_1a_0 \tdeux X_1X_0+a_1^2a_0\ttroisdeux X_1^2X_0+a_2a_0^2\ttroisun X_2X_0^2\\
&+a_1^3a_0\tquatrecinq X_1^3X_0+a_2a_1a_0^2\left(2\tquatredeux+\tquatrequatre\right)X_2X_1X_0^2+a_3a_0^3\tquatreun X_3X_0^3\\
&+a_1^4a_0\tcinqquatorze X_1^4X_0+a_2a_1^2a_0^2\left(\tcinqcinq+2\tcinqhuit+2\tcinqonze+\tcinqtreize\right)X_2X_1^2X_0^2+2a_2^2a_0^2 \tcinqsix X_2^2X_0^2\\
&+a_3a_1a_0^3\left(3\tcinqdeux+ \tcinqdix\right)X_3X_1X_0^3+a_4a_0^4 \tcinqun X_4X_0^4+\ldots
\end{align*}\end{example}

A direct induction proves the following lemma:

\begin{lemma}\label{lem3.20}
For any $\alpha\in \Lambda$,
\[T_\alpha=  \left(\prod_{i=0}^\infty a_i^{\alpha_i}\right) \sum_{T\in \calT, \: M(T)=\alpha}p_TT=\frac{1}{c_\alpha}\left(\prod_{i=0}^\infty a_i^{\alpha_i}\right)\Psi(x^\alpha).\]
\end{lemma}

\begin{remark}\label{rk3.6}
This formula can be rewritten as
\[T_\alpha=\frac{1}{\alpha!} \left(\prod_{i=0}^\infty (i!a_i)^{\alpha_i}\right) \Psi(x^\alpha).\]
In particular, when $f=\exp(h)=\displaystyle \sum_{k=0}^\infty \frac{h^k}{k!}$, this gives
\[T_\alpha=\frac{1}{\alpha!} \Psi(x^\alpha).\]
\end{remark}

\begin{theo}\label{theo3.21}

Let $f=\displaystyle \sum_{k=0}^\infty a_kh^k\in \mathbb{C}[[h]]$ be a formal series. Then $H_f$ is a Hopf subalgebra of $\HCK$ if, and only if, $a_0=0$ or if for any $i\geq 0$,
\[a_i=0\Longrightarrow a_{i+1}=0.\]
If all the $a_i$'s are nonzero, then $H_f=\im(\Psi)$.
\end{theo}

\begin{proof}
If $a_0=0$, as any rooted tree has at least one leaf (that is, a vertex of fertility $0$), by Lemma \ref{lem3.20},
$\scrT=0$, and $H_f=\mathbb{C}$. We now assume that $a_0\neq 0$. \\

$\Longrightarrow$. Let us assume that $H_f$ is a Hopf subalgebra of $\HCK$, and that $a_{i+1}\neq 0$. Then 
\[\scrT_{(i+1,0,\ldots,0,1,0,\ldots)}=a_0^{i+1}a_{i+1}B^+(\tun^{i+1})=a_0^{i+1}a_{i+1}C_{i+2}\in H_f.\]
As $a_0,a_{i+1}\neq 0$, $C_{i+2}\in H_f$. As $H_f$ is a Hopf subalgebra, $\Delta(C_{i+2})\in H_f \otimes H_f$, which implies that $C_{i+1}\in H_f$. Therefore, $\scrT_{(i,0,\ldots,0,1,0,\ldots)}\neq 0$, so
\[\scrT_{(i,0,\ldots,0,1,0,\ldots)}=a_0^ia_iC_{i+1}\neq 0,\]
and finally $a_i\neq 0$. By transposition, if $a_i=0$, then $a_{i+1}=0$. \\

$\Longleftarrow$. Let us assume that all the $a_i$'s are nonzero. Then $H_f$ is the subalgebra of $\HCK$ generated by the elements $\Psi(x^\alpha)$
where $\alpha$ runs in $\Lambda$: this is precisely $\im(\Psi)$, which is indeed a Hopf subalgebra of $\HCK$. 

Let us now assume that $a_0,\ldots,a_k\neq 0$, and $a_i=0$ if $i>k$. We put
\[\Lambda_k=\{\alpha \in \Lambda\mid \forall i\in \N,\: \alpha_i\leq k\}.\]
Then $H_f$ is the subalgebra generated
by the elements $\Psi(x^\alpha)$, where $\alpha \in \Lambda_k$. As the $x^\alpha$, $\alpha \in \Lambda_k$, generate a Hopf subalgebra of $S(\algdual)$, $H_f$ is a Hopf subalgebra of $\HCK$.
\end{proof}

\section{A few consequences}

\subsection{Fundamental polynomial invariant}

By Theorem \ref{theo1.4}, as $S(\algdual)$ is a connected double bialgebra, there exists a unique double bialgebra morphism $\phiMI$ from $(S(\algdual),\mid,\DeltaNMI,\deltaNMI)$ to $(\K[X],m,\Delta,\delta)$.

\begin{lemma}
$\phiMI=\phiCK\circ \Psi$.
\end{lemma}

\begin{proof}
By composition, $\phiCK\circ \Psi$ is a double bialgebra morphism from $S(\algdual)$ to $\K[X]$. By unicity, it is equal to $\phiMI$. 
\end{proof}

The Dyson-Schwinger equation (\ref{EQ6}) allows to inductively compute the polynomial $\phiMI(x^\alpha)$:

\begin{prop} \label{prop4.2}
Let us put
\[\calP=\sum_{\alpha \in \Lambda} \phiMI(x^\alpha)\dfrac{X^\alpha}{\alpha!} \in \K[X][[X_i\mid i\in \N]].\]
It satisfies the fixed-point equation
\[\calP=L\left(\sum_{i=0}^\infty \frac{X_i}{i!}\calP^i\right).\]
\end{prop}

\begin{proof}
Let us consider Equation (\ref{EQ6}) associated to the formal series $f=\exp(h)$.
By Lemma \ref{lem3.20} and Remark \ref{rk3.6}, this gives
\[\scrT=\sum_{\alpha \in \Lambda} \Psi(x^\alpha)\frac{X^\alpha}{\alpha!}.\]
Hence,
\begin{align*}
\phiCK(\scrT)&=\sum_{\alpha \in \Lambda} \phiCK \circ \Psi(x^\alpha)\frac{X^\alpha}{\alpha!}=\sum_{\alpha \in \Lambda} \phiMI(x^\alpha)\frac{X^\alpha}{\alpha!}=\calP.
\end{align*}
Let us apply $\phiCK$ to both sides of (\ref{EQ6}).
\begin{align*}
\calP&=\phiCK(\scrT)=\phiCK \circ B^+\left(\sum_{i=0}^\infty \frac{X_i}{i!}\scrT^i\right)=L\circ \phiMI\left(\sum_{i=0}^\infty \frac{X_i}{i!}\scrT^i\right)\\
&=L\left(\sum_{i=0}^\infty \frac{X_i}{i!}\phiMI(\scrT)^i\right)=L\left(\sum_{i=0}^\infty \frac{X_i}{i!}\calP^i\right). \qedhere
\end{align*}\end{proof}

\begin{example} \label{ex4.1}
\begin{enumerate}
\item This fixed-point equation allows to inductively compute $\phiMI(x^\alpha)$. For example, we obtain
\begin{align*}
\phiMI(x_1x_0)&=\frac{1}{2}(X-1)X,&\mbox{OEIS }&A000217\\
\phiMI(x_2x_0^2)&=\frac{1}{6}(2X-1)(X-1)X,&\mbox{OEIS }&A000330\\
\phiMI(x_1^2x_0)&=\frac{1}{3}(X-1)(X-2)X,&\mbox{OEIS }&A000292\\
\phiMI(x_3x_0^3)&=\frac{1}{4}(X-1)^2X^2,&\mbox{OEIS }&A000537\\
\phiMI(x_2x_1x_0^2)&=\frac{1}{6}(2X-1)(X-1)(X-2)X,&\\
\phiMI(x_1^3x_0)&=\frac{1}{4}(X-1)(X-2)(X-3)X,&\mbox{OEIS }&A033487\\
\phiMI(x_4x_0^4)&=\frac{1}{30}(3X^2-3X-1)(2X-1)(X-1)X,&\mbox{OEIS }&A000538\\
\phiMI(x_3x_1x_0^3)&=\frac{1}{120}(42X^2-39X-1)(X-1)(X-2)X,&\\
\phiMI(x_2^2x_0^3)&=\frac{1}{20}(8X^2-11X+1)(X-1)(X-2)X,&\\
\phiMI(x_2x_1^2x_0^2)&=\frac{1}{60}(11X-29)(2X-1)(X-1)(X-2)X,&\\
\phiMI(x_1^4x_0)&=\frac{1}{5}(X-1)(X-2)(X-3)(X-4)X.&\mbox{OEIS }&A158874
\end{align*}
\item Let $n\geq 1$. As $\Psi(x_1^{n-1}x_0)=(n-1)! L_n$, we obtain
\[\phiMI(x_1^{n-1}x_0)=(n-1)!\phiCK(L_n)=(n-1)!H_n(X)=\frac{X(X-1)\ldots (X-n+1)}{n}.\]
As $\Psi(x_{n-1}x_0^{n-1})=C_n$, we obtain
\[\phiMI(x_{n-1}x_0^{n-1})=\phiCK(C_n)=\phiCK\circ B^+(\tun^{n-1})=L\circ \phiCK(\tun^{n-1})=L(X^{n-1}).\]
Therefore, for any $k\geq 1$,
\[\phiMI(x_{n-1}x_0^{n-1})(k)=0^{n-1}+\cdots+(k-1)^{n-1}.\]
 By Faulhaber's formula, 
\[\phiMI(x_{n-1}x_0^{n-1})=\frac{1}{n} \sum_{i=0}^{n-1}(-1)^i\binom{n}{i} B_iX^{n-i},\]
where $(B_k)_{k\geq 0}$ is the sequence of Bernoulli numbers.
\[\begin{array}{|c||c|c|c|c|c|c|c|c|c|c|c|c|c|}
\hline k&0&1&2&3&4&5&6&7&8&9&10&11&12\\
\hline \hline &&&&&&&&&&&&&\\[-3mm] B_k&1&\dfrac{1}{2}&\dfrac{1}{6}&0&-\dfrac{1}{30}&0&\dfrac{1}{42}&0&-\dfrac{1}{30}&0&\dfrac{5}{66}&0&-\dfrac{691}{2730}\\[3mm]
\hline\end{array}\]
\end{enumerate}\end{example}

\subsection{The antipode}

We use Theorem \ref{theo1.3} to give a formula for the antipode of $(S(\algdual),\mid,\DeltaNMI)$. We first need a way to compute $\phiMI(x^\alpha)(-1)$ for $\alpha\in \Lambda$.

\begin{lemma}\label{lem4.3}
For any $P\in \K[X]$, $L(P)(-1)=-P(-1)$.
\end{lemma}

\begin{proof}
By linearity in $P$, it is enough to prove this when $P$ is the $n$-th Hilbert polynomial $H_n$. For any $k\geq 0$,
\[H_k(-1)=\frac{(-1)(-2)\ldots(-n)}{n!}=(-1)^n.\]
Therefore,
\[L(H_n)(-1)=H_{n+1}(-1)=(-1)^{n+1}=-H_n(-1). \qedhere\]
\end{proof}

\begin{prop}
We define the character $\muMI$ of $S(\algdual)$ by
\begin{align*}
&\forall \alpha\in \Lambda,&\muMI(x^\alpha)&=\phiMI(x^\alpha)(-1).
\end{align*}
We consider the formal series
\[U=\sum_{\alpha\in \Lambda} \muMI(x^\alpha)\frac{X^\alpha}{\alpha!} \in \K[[X_0,X_1,\ldots]].\]
Then $U$ satisfies the fixed-point equation
\begin{align*}
U&=-\sum_{i=0}^\infty \frac{X_i}{i!}U^i.
\end{align*}\end{prop}

\begin{proof}
We use Proposition \ref{prop4.2}. Then $U=\calP(-1)$, so, by Lemma \ref{lem4.3},
\[U=\calP(-1)=L\left(\sum_{i=0}^\infty \frac{X_i}{i!}\calP^i\right)(-1)=-\sum_{i=0}^\infty \frac{X_i}{i!} \calP(-1)^i=-\sum_{i=0}^\infty \frac{X_i}{i!}U^i. \qedhere\]
\end{proof}

\begin{example}\begin{enumerate}
\item This fixed-point equation allows to compute $\muMI(x^\alpha)$ by induction. 
 \begin{align*}
\begin{array}{|c||c|c|c|c|c|c|c|c|c|c|c|}
\hline&&&&&&&&&&&\\[-3mm]
 x^\alpha&x_1x_0&x_2x_0^2&x_1^2x_0&x_3x_0^3&x_2x_1x_0^2&x_1^3x_0&x_4x_0^4&x_3x_1x_0^3&x_2^2x_0^3&x_2x_1^2x_0^2&x_1^4x_0\\
\hline\hline&&&&&&&&&&&\\[-3mm]
\muMI(x^\alpha)&1&-1&-2&1&3&6&-1&-4&-6&-12&-24\\
\hline\end{array}
\end{align*}
\item Let $n\geq 1$. From Example \ref{ex4.1}, 
\[\muMI(x_1^{n-1}x_0)=\frac{(-1)(-2)\ldots (-n)}{n}=(-1)^n (n-1)!.\]
Moreover, by Lemma \ref{lem4.3},
\[\muMI(x_{n-1}x_0^{n-1})=L(X^{n-1})(-1)=-(-1)^{n-1}=(-1)^n.\] 
\end{enumerate}\end{example}

\begin{prop}\label{prop4.5}
We denote by $S$ the antipode of $(S(\algdual),\mid,\DeltaNMI)$. For any $\alpha\in \Lambda$,
\begin{align*}
S(x^\alpha)&=\sum_{\substack{\beta \in \Lambda,\\ \alpha=\alpha_1+\cdots+\alpha_{\ell(\beta)}}} \frac{\muMI(x^\beta)}{\beta!}\frac{\alpha!}{\alpha_1!\ldots \alpha_{\ell(\alpha)}!}x^{\alpha_1}\mid\cdots\mid x^{\alpha_{\beta_0}}\mid
D(x^{\beta_0+1})\mid\cdots \mid D(x^{\beta_0+\beta_1})\mid D^2(x^{\beta_0+\beta_1+1})\mid\cdots
\end{align*}\end{prop}

\begin{proof}
We apply Theorem \ref{theo1.3}. This gives, with Proposition \ref{prop3.11},
\begin{align*}
S(x^\alpha)&=\sum_{\beta \in \Lambda} \frac{\muMI(x^\beta)}{\beta!}\mid^{(\ell(\beta)-1)}\circ\left(\bigotimes_{n=0}^\infty {D'^n}^{\otimes \beta_n} \right)\circ \deltash^{(\ell(\beta)-1)}(x^\alpha).
\end{align*}
The result comes from 
\begin{align*}
&\forall n\geq 1,&\deltash^{(n-1)}(x^\alpha)&=\sum_{\alpha=\alpha_1+\cdots+\alpha_n}\frac{\alpha!}{\alpha_1!\ldots \alpha_n!}x^{\alpha_1}\otimes\cdots \otimes x^{\alpha_n}. \qedhere
\end{align*}\end{proof}

\bibliographystyle{amsplain}
\bibliography{biblio}

\end{document}